\documentclass[11pt,a4paper]{amsart}
\usepackage[left=3cm,right=3cm,top=2cm,bottom=2cm, marginparwidth=4cm, marginparsep=0.5cm]{geometry} 
\usepackage{amsmath,amssymb,mathtools,mathdots} 
\usepackage{amsthm}
\usepackage{graphicx}
\usepackage{mathrsfs}
\usepackage{bbm}
\usepackage{fancyhdr}

\usepackage{marginnote} 
\usepackage[utf8]{inputenc}         
\usepackage[T1]{fontenc}            %
\usepackage[final]{microtype} 
\usepackage{url}                    

\usepackage{tikz-cd}                   
\usetikzlibrary{arrows}                 
\usepackage{adjustbox}
\usepackage[]{color} 
\definecolor{azure}{rgb}{0.0, 0.5, 1.0}
\definecolor{awesome}{rgb}{1.0, 0.13, 0.32}
\usepackage{framed} 
\usepackage[hidelinks]{hyperref}
\hypersetup{
  colorlinks   = true, 
  urlcolor     = azure, 
  linkcolor    = azure, 
  citecolor    = awesome 
}
\usepackage[capitalise]{cleveref}

\tikzset{
commutative diagrams/.cd,
arrow style=tikz,
diagrams={>=latex}} 

\usetikzlibrary{decorations.markings}
\makeatletter
\tikzcdset{
  open/.code     = {\tikzcdset{hook, circled};},
  closed/.code   = {\tikzcdset{hook, slashed};},
  open'/.code    = {\tikzcdset{hook', circled};},
  closed'/.code  = {\tikzcdset{hook', slashed};},
  circled/.code  = {\tikzcdset{markwith = {\draw (0,0) circle (.375ex);}};},
  slashed/.code  = {\tikzcdset{markwith = {\draw[-] (-.4ex,-.4ex) -- (.4ex,.4ex);}};},
  markwith/.code ={
    \pgfutil@ifundefined%
    {tikz@library@decorations.markings@loaded}%
    {\pgfutil@packageerror{tikz-cd}{You need to say %
      \string\usetikzlibrary{decorations.markings} to use arrows with markings}{}}{}%
    \pgfkeysalso{/tikz/postaction = {
      /tikz/decorate,
      /tikz/decoration={markings, mark = at position 0.5 with {#1}}}
    }
  },
}
\makeatother

\usepackage{enumitem} 
\setlist[itemize]{noitemsep, nolistsep}
\setlist[enumerate]{noitemsep, nolistsep}

\newtheorem{thm}{Theorem}[section]
\newtheorem{prop}[thm]{Proposition}

\newtheorem{lm}[thm]{Lemma}
\newtheorem{cor}[thm]{Corollary}

\theoremstyle{definition}
\newtheorem{defn}[thm]{Definition}
\newtheorem{ex}[thm]{Example}

\theoremstyle{remark}
\newtheorem{rke}[thm]{Remark}

\newcommand{\p}{\mathfrak{p}}

\newcommand{\m}{\mathfrak{m}}
\newcommand{\ie}{\emph{i.e.,} }
\newcommand{\eg}{\emph{e.g.,} }
\newcommand{\cf}{\emph{cf.} }

\newcommand{\K}{\mathbb{K}}
\newcommand{\R}{\mathbb{R}}
\newcommand{\C}{\mathbb{C}}
\newcommand{\Q}{\mathbb{Q}}
\newcommand{\Z}{\mathbb{Z}}
\newcommand{\N}{\mathbb{N}}
\newcommand{\Pj}{\mathbb{P}}

\renewcommand{\mod}{\:\mathrm{mod}\:}

\newcommand{\Or}{\mathcal{O}}

\newcommand{\A}{\mathbb{A}}
\newcommand{\ord}{\mathrm{ord}}
\newcommand{\gr}{\mathrm{gr}}
\newcommand{\Gr}{\mathrm{Gr}}
\newcommand{\init}{\mathrm{in}}
\newcommand{\Init}{\mathrm{In}}
\newcommand{\IN}{\mathrm{IN}}

\title{Fibers of phase tropicalizations}
\author{Andrei Bengu\textcommabelow s-Lasnier and Mikhail Shkolnikov
}

\begin{document}

\pagestyle{fancy}
\fancyhf{}
\fancyhead[CE]{\textsc{Andrei Bengu\textcommabelow s-Lasnier and Mikhail Shkolnikov}}
\fancyhead[CO]{\textsc{Fibers of phase tropicalizations}}
\fancyhead[LE,RO]{\thepage}


\address{Andrei Bengu\textcommabelow s-Lasnier\\
Institute of Mathematics and Informatics\\
Bulgarian Academy of Sciences\\
Akad. G. Bonchev, Sofia 1113, Bulgaria \\
email:abengus@math.bas.bg}

\address{Mikhail Shkolnikov\\
Institute of Mathematics and Informatics\\
Bulgarian Academy of Sciences\\
Akad. G. Bonchev, Sofia 1113, Bulgaria \\
email:m.shkolnikov@math.bas.bg}

\begin{abstract}
The subject of the present paper is phase tropicalization, which was used crucially
in the context of Mikhalkin's correspondence theorem for curve counting in the complex
coefficient case. The subject can be traced back to Viro's patchworking for constructing
topological types of real algebraic curves. These two instances correspond to complex
and real phases. Both fall into the category of what can be called "abelian" or
classical tropicalization, referring to degenerations of varieties within an algebraic
torus (or its compactification). In contrast, in "non-abelian" tropicalizations
the ambient torus is replaced by a non-commutative group such as the special linear
group. This is the beginning of a general theory valid for a wide array of coefficient systems
and dimensions. As an application, the paper settles the question of phase tropicalization for
the special linear group $\mathrm{SL}_2$. It also gives an algebraic explanation and phase
extension of the case of curves, previously studied in the purely geometric framework.
To accomplish these tasks we introduce valuative tools that allow us to prove an affine version
of Kapranov's theorem on tropical hypersurfaces and its generalization to arbitrary tropical
varieties. Most notably, we show the functorial properties of the graded ring of a valuation and
exhibit the polynomial structure of the graded ring of monomial valuations.
\end{abstract}


\maketitle

\tableofcontents

\section{Introduction}\label{sec:intro}

Kapranov's tropical geometric theorem is a foundational result relating tropicalizations
of toric varieties to their defining ideals (see for instance \cite[3.1.3 and 3.2.3]{MS}).
In fact the theorem is stronger and gives the defining ideal of phase tropicalizations.
In this paper we extend these results to affine varieties in the following way. Consider an
algebraic variety $X$ defined over the field of (reversed) Hahn series $(\K,\nu)$, where $\nu$
is its associated valuation. Elements in this field can be written as $ct^\alpha+o(t^\alpha)$,
where $c\in\C,\:\alpha\in\R$ and the term $o(t^\alpha)$ designates lower order terms. Any vector
$A\in\K^n\setminus\{0\}$ can be written as $t^\alpha B+o(t^\alpha),\:B\in\C^n\setminus\{0\}$ by
taking the entries of highest order. Define $\Init_\nu$ as the map that associates to this vector
its leading part $(\alpha, B)$. Our main result states that these leading terms form explicit
algebraic sets.

For any $f=\sum_uc_ux^u\in\K[x_1,\ldots,x_n]$, set $c_u=\lambda_ut^{\alpha_u}+o(t^{\alpha_u})$ and
write $\nu_\alpha(f):=\max\{\nu(c_u)+\alpha|u|\::\:c_u\neq0\}$, where
$|(u_1,\ldots,u_n)|=u_1+\cdots+u_n$. Define then
\[\IN_\alpha(f):=\sum_{\nu(c_u)+\alpha|u|=\nu_\alpha(f)}\lambda_uX^u\in\C[X_1,\ldots,X_n].\]
For any ideal $I$ in $\K[x_1,\ldots,x_n]$, let $\IN_\alpha(I)$ be the ideal in
$\C[X_1,\ldots,X_n]$ generated by the $\IN_\alpha(f),\,f\in I$. Recall that the disjoint union
of a family of sets $(S_j)_{j\in J}$ is constructed as $\bigsqcup_{j\in J}S_j:=
\bigcup_{j\in J}\{j\}\times S_j$. The following result is the main contribution of this paper.

\begin{thm}\label{thm:main}
Consider an algebraic variety $X\subset\K^n\setminus\{0\}$, given by an ideal
$I\subset\K[x_1,\ldots,x_n]$.
\begin{enumerate}
\item The set of initial parts can be described as a fibration of algebraic sets
\[\Init_\nu(X)=\bigsqcup_{\alpha\in\R}\mathbb{V}(\IN_\alpha(I))\setminus\{0\}.\]
\item There exists a set of critical values $\beta_0<\ldots<\beta_r$ such that, for
$i=0,\ldots,r+1$ the function $\alpha\in(\beta_{i-1},\beta_i)\mapsto\IN_\alpha(I)$ is
constant and gives a homogeneous ideal. The ideals $\IN_{\beta_i}(I)$ are not homogeneous
(we set $\beta_{-1}=-\infty$ and $\beta_{r+1}=+\infty$).
\end{enumerate}
\end{thm}

Our main application of this result concerns $\mathrm{SL}_2$ (or $\mathrm{PSL}_2$)
tropicalizations as they were introduced in \cite{MS22}. Simplifying, they associate
to a family of varieties $X=(X_t)_{t>0}$ with $X_t\subset\mathrm{PSL}_2(\C)$ a limit degeneration
that we denote $\mathrm{VAL}(X)\subset\mathrm{PSL}_2(\C)$. Consider the polar decomposition of
$\mathrm{PSL}_2(\C)\simeq\mathbb{H}^3\times\mathrm{PSU}(2)$. The set $\mathbb{H}^3$ of positive
definite, Hermitian and unimodular matrices is a (Hermitian) model of the hyperbolic space.
We write $O$ for its center, which is the identity matrix. The natural metric gives rise to a
distance function that can be calculated explicitly: $\mathrm{dist}(O,B)=|\ln(\lambda(B))|$,
where $\lambda(B)$ is one of the eigenvalues of $B$. Focusing on the Hermitian parts
of $\mathrm{VAL}(X)$, we obtain the hyperbolic tropicalization of $X$. The authors of \cite{MS22}
obtain that the hyperbolic tropicalizations of families of complex algebraic curves are unions of
concentric spheres around $O$ and segments extending to geodesics passing through $O$. In
\cite{PS25}, it is shown that the hyperbolic tropicalization of a family of surfaces $(X_t)_{t>0}$
in $\mathrm{PSL}_2(\C)$ is a complement to an open ball in $\mathbb{H}^3$ centered at the image of
the identity matrix denoted by $O$.

In this work, \cref{thm:main} completes \cite[Thm. 3, 4]{SP25}, thus proving an
$\mathrm{SL}_2$ version of Kapranov's theorem at the algebraic level for complex surfaces, and
reprove and refine the analogous theorem for curves. Write $\widetilde{\mathrm{VAL}}(X)$ for the
$\mathrm{SL}_2$ tropicalization of $X$. When restricting our varieties to convergent Hahn series,
one can still see $X$ as a family of varieties
\footnote{More specifically, suppose $X$ is given by equations $X=\mathbb{V}(f_1,\ldots,f_s)$.
Each polynomial $f_j$ has coefficients in $\K$, thus they are functions in $t$, so we
can write $f_j=f_j(t)$. We can thus define the family $(X_t)_{t>0}$ by
$X_t=\mathbb{V}(f_1(t),\ldots,f_s(t))$.}. In this situation, the critical levels
correspond to the radii of the concentric spheres in $\mathbb{H}^3$ of the hyperbolic
tropicalization.

To help in describing the hyperbolic degenerations for surfaces we upgrade our hyperbolic picture
to a double hyperbolic one. More precisely, consider the two polar decompositions of a matrix
$C\in\mathrm{SL}_2(\C)$, $C=PU$ and its dual $B=U'P'$ with $P^*=P$ positive definite, unimodular,
$UU^*=U^*U=1$ and likewise for $P',U'$. Then associate to $C$, its two Hermitian parts $(P,P')$
and denote it $\widehat{\varkappa}(C)$. We prove the following in \cref{sec:layered}.

\begin{thm}\label{thm:surfaces}
Consider a surface $X\subset\mathrm{SL}_2(\K)$ and let $\beta_0<\beta_1<\dots<\beta_r$ be its
critical levels. Then the double-hyperbolic tropicalization is
\[\widehat{\varkappa}(\widetilde{\mathrm{VAL}}(X))=\bigcup_{i=0}^r\{\beta_i\}\times Q_2(\mathbb{C})
\cup\bigcup_{k=0}^r(\beta_i,\beta_{i+1})\times C_k,\]
where $\beta_{r+1}=+\infty$ and $C_0,\dots,C_r$ are complex algebraic curves on
$Q_2(\C)$ of symmetric bi-degree increasing with $i$.
\end{thm}

To summarize, this work settles the question of how  phase tropical limits of surfaces in
$\mathrm{SL}_2(\C)$ may look like in general. In particular we resolve the previous issue
of showing that inclusions in Theorems 3 and 4 of \cite{SP25} are actually equalities.
This opens the door to extending the general principle that phase tropicalizations restore
the topology of the initial variety. The first instance of this idea was Viro's patchworking
\cite{V83}, which, in essence, is a real phase toric tropicalization. The second instance
was accomplished in theorems for complex phases, of Kerr and Zharkov \cite{KZ18},
and of Kim and Nisse \cite{KN21}. The $\mathrm{SL}_2$ versions of these facts are subjects
of current investigation.

We reiterate that what follows is building foundations of a more general framework
which may be referred to as {\it affine initial forms}, inspired by and applied to
$\mathrm{SL}_2$ phase tropicalization, but not limited to it. The closest set of examples
are affine quadrics as ambient spaces with the first relevant instance being akin to
Brugallé's conic floor diagrams \cite{B15}. Nevertheless in order to generalize the above
statements to other groups, one would be required to go beyond the first-order term in the
asymptotic expansion (see \cref{ex:counter-ex-dim3}). This necessitates an adequate analytic,
geometric, representation, and valuation theoretic sophistication that has yet to be developed.

\begin{rke}
After the first version of the article appeared on arXiv, a similar generalization of Kapranov's
theorem and of the fundamental theorem of tropical geometry has been communicated to us.
Indeed in \cite{MaSm}, tropical varieties are defined via hyperfield arithmetic and it is proven
that a valuation is a particular type of hyperfield morphism. Their notion of fine tropical
variety appears to be related to our sets $\Init_\nu(\mathbb{V}(I))$. However, at the moment of
writing, it is unclear to us how to precisely relate these two notions and thus how to recover
our \cref{thm:main} from the results of Maxwell and Smith.
\end{rke}

The paper is organized as follows. In \cref{sec:fibers} we detail different types of
tropicalizations that help clarify the general picture of the hyperbolic degenerations.
In \cref{sec:prelim} and \cref{sec:composing_init} we introduce the elements of valuation
that are needed for the proof of \cref{thm:main}. Notably we make extensive use of
the graded algebra associated to a valuative pair. Additionally, we introduce a similar
construction, that we name phase space. This is the structure inside which the initial terms
of vectors live. It is more convenient for our proofs and future applications to establish
the formalism in full generality. One important property of graded algebras is functoriality.
Roughly speaking, functoriality allows one to make a change of variables. It is the key
ingredient in our arguments towards the Kapranov-like theorem. Our proof is divided between
\cref{sec:lifting} and \cref{sec:layered}. More precisely, the first point of \cref{thm:main}
is shown in \cref{thm:phase_trop} and the second is contained in \cref{prop:crit_values}.
We sketch an alternative proof to \cref{thm:phase_trop} in \cref{app:alt_proof} and in
\cref{app:diffeo} we detail the proof of \eqref{eq:diffeo} and show how it defines a diffeomorphism
between $\mathrm{SL}_2(\C)$ and its non-abelian tropicalization.

\textbf{Funding.} This work has been supported by a "Peter Beron i NIE"
fellowship [KP-06-DB-5] from the Bulgarian Science Fund for the first author, and by
the Simons Foundation, grant [SFI-MPS-T-Institutes-00007697],
the Ministry of Education and Science of the Republic of Bulgaria,
grant [DO1-239/10.12.2024], as well as by the National Science Fund,
The Ministry of Education and Science of the Republic of Bulgaria, under contract [KP-06-N92/2]
for the second author.

\textbf{Acknowledgments.} We would like to extend our gratitude to Peter Petrov, Grigory Mikhalkin
and Ilia Zharkov for proposing to us this project, their ample support and fruitful discussions.
The second author gratefully acknowledges the hospitality of IMPA (Rio de Janeiro),
where this work was completed during his visit.

\textbf{Conflict of interest.} The authors have no conflicts of interest to declare.

\section{Fibers of \texorpdfstring{$\mathrm{SL}_2$}{SL2}
phase tropicalization}\label{sec:fibers}

Before diving into the principal algebraic narrative of this article, we would like to detail
our main geometric application. We start by defining the $\mathrm{SL}_2$
degenerations. Consider a family of matrices $(A_t)_{t>0},\:A_t\in\mathrm{SL}_2(\C)$. We see it
as a function $\R_{>0}\to\mathrm{SL}_2(\C)$, and we assume it has a first asymptotic
$t^\alpha B+o(t^\alpha),\:B\in\mathrm{Mat}_2(\C)$ for large $t$. For $h>0$ fix the operator
$R_h:\mathrm{SL}_2(\C)\to\mathrm{SL}_2(\C)$ that acts in the following way: for
$A\in\mathrm{SL}_2(\C)$ consider its polar decomposition $A=PU$ where $P=P^*$ and
$UU^*=U^*U=1$; then $R_h(A)=P^hU$. Our geometric degeneration of $(A_t)_t$ is the limit
\[\lim_{t\to+\infty}R_{\log(t)^{-1}}(A_t)\in\mathrm{SL}_2(\C).\]
We write this limit as $\widetilde{\mathrm{VAL}}(A)$. For a family of varieties $(X_t)_{t>0},\:
X_t\subset\mathrm{SL}_2(\C)$ one defines its $\mathrm{SL}_2$ \textit{tropicalization} as
$\widetilde{\mathrm{VAL}}(X)=\{\lim_{t\to\infty}R_{\log(t)^{-1}}(A_t):\:A_t\in X_t,\:t>0\}$.

In \cref{app:diffeo} we prove the following formula.
\begin{equation}\label{eq:diffeo}
\lim_{t\to+\infty}R_{\log(t)^{-1}}(A_t)=[\![e^\alpha B+e^{-\alpha}(B^*)^{adj}]\!],
\end{equation}
where $[\![C]\!]=(\det(C))^{-1/2}C\in\mathrm{SL}_2(\C)$ is the unique normalization assuming
$\det(C)\in\mathbb{R}_{>0}$. A similar formula for $\mathrm{PSL}_2(\C)$ appears first in
\cite{SP24}, with the proof sketched in \cite{SP25}. This expression is hard to work with directly,
nevertheless one first remarks that it only depends on the first asymptotics of the entries in
$(A_t)_t$. It is thus well-defined, regardless of whether the entries of $(A_t)_t$ are convergent
Hahn series or general formal power series. Thus, instead of considering families of points
$(A_t)_t$ or families of varieties $(X_t)_t$, where the entries of the matrices or coefficients
of the defining polynomials vary with a parameter $t$, one can consider points and varieties with
entries and coefficients in a valued field $\K$. The typical prototype for $\K$ is the field of
Hahn series, converging or not. Thus from now on, we assume $A\in\mathrm{SL}_2(\K)$ and
$X\subset\mathrm{SL}_2(\K)$ is a subvariety. We can still define $\widetilde{\mathrm{VAL}}(X)$
via the right-hand side expression of \eqref{eq:diffeo}.

We give a geometric picture of this process.
Let $\mathcal{S}$ denote the total space of a circle bundle over the quadric surface
$Q_2(\C)=\{[B]_{\C^*}:\:\det(B)=0\}\subset\mathbb{CP}^3$, with the fiber over a point
$[B]_{\C^*}$ being $\{[cB]_{\mathbb{R}_{>0}}:c\in U(1)\}$. In \cref{app:diffeo} we show
there is a diffeomorphism $(0,+\infty)\times\mathcal{S}\to\mathrm{SL}_2(\C)/\mathrm{SU}(2)$
given by $(\alpha,[B]_{\mathbb{R}_{>0}})\mapsto[\![e^\alpha B+e^{-\alpha}(B^*)^{adj}]\!]$.
This allows us to see $\mathrm{SL}_2(\C)$ as a disjoint union
$\big(\{0\}\times\mathrm{SU}(2)\big)\cup\big((0,\infty)\times\mathcal{S}\big)$.
Our situation is thus dramatically simplified, since we can now focus on
$(\alpha,[B]_{\mathbb{R}_{>0}})$ instead of $\widetilde{\mathrm{VAL}}$.

We abbreviate $\{B\in\mathrm{Mat}_2(\C)\setminus\{0\}:\:\det(B)=0\}$ with $\{\det=0\}$.
The $\mathrm{SL}_2$ tropicalization is dominated via a surjection
\[\pi_{\mathbb{R}}:\Big(\{0\}\times\mathrm{SL}_2(\C)\Big)\cup
\Big((0,\infty)\times\{\det=0\}\Big)\to
\Big(\{0\}\times\mathrm{SU}(2)\Big)\cup
\Big((0,\infty)\times\mathcal{S}\Big),\]
that is defined as follows. Consider the polar decomposition of the special
linear group: $\mathrm{SL}_2(\C)\simeq\mathbb{H}^3\times\mathrm{SU}(2)$.
For any matrix $B\in\mathrm{SL}_2(\C)$ consider its polar decomposition $B=PU$
where $P\in\mathbb{H}^3$ and $U\in\mathrm{SU}(2)$. Set
\[\pi_\R(0,B)=(0,U)\quad\text{and}\quad
\pi_\R(\alpha,B)=(\alpha,[B]_{\mathbb{R}_{>0}})\quad\text{for }\alpha>0.\]
We can thus shift our focus to this new set
\[\Big(\{0\}\times\mathrm{SL}_2(\C)\Big)\cup
\Big((0,\infty)\times\{\det=0\}\Big)\subset\R\times\C^4.\]
This can be seen as the first instance of the new type of phase tropicalization of
$\mathrm{SL}_2(\K)$ for $\mathbb{K}$ being a complex-coefficient real-powered series.
We will henceforth call it \textit{valuative tropicalization}.

It was mentioned in the introduction, that $\mathrm{SL}_2$ tropicalizations of surfaces
are almost trivial to describe: they are complements to open balls $B_r(O)$ of radius
$r\geq 0$ centered in $O\in\mathbb{H}^3$. Thus, one is motivated to look at a
corresponding phase tropicalization $X\subset\mathrm{SL}_2(\C)$ fibered over
$\mathbb{H}^3\setminus B_r(O)$. Via the polar decomposition we can identify the
quotient $\mathrm{SL}_2(\C)/\mathrm{SU}(2)$ with $\mathbb{H}^3$. Indeed, if $B=PU$
is a polar decomposition then $P^2=BB^*$. Since $P\mapsto P^2$ is a diffeomorphism on
positive definite matrices, one can consider the map $\varkappa:\mathrm{SL}_2(\C)\to
\mathbb{H}^3,B\mapsto BB^*$ as a fibration that we call the \textit{hyperbolic amoeba map}.
We call the images of tropicalizations through this map \textit{hyperbolic tropicalizations}.

To better see the global picture of the phase tropicalization, we propose the
following change of perspective on the ambient space. So far, we were dealing with
the projection $\varkappa\colon\mathrm{SL}_2(\C)\to\mathbb{H}^3$ which arose from
taking the right quotient by the maximal compact subgroup $\mathrm{SU}(2)$.
If we consider the left quotient $\mathrm{SU}(2)\setminus\mathrm{SL}_2(\C)$
instead, we are led to defining the dual fibration $\varkappa^*
\colon\mathrm{SL}_2(\C)\to\mathbb{H}^3,BB\mapsto B^*B$. What one observes,
is that the distances from $O$ to $\varkappa(A)$ and to $\varkappa^*(A)$ coincide.
Thus all possible pairs $\widehat\varkappa(A)=(\varkappa(A),\varkappa^*(A))$
belong to a five-dimensional subspace $\mathcal{C}$ of
$\mathbb{H}^3\times\mathbb{H}^3$, which may be identified with a real cone over
$(\partial\mathbb{H}^3)^2=\mathbb{CP}^1\times\mathbb{CP}^1=Q_2(\C)$. It is natural
to fix the map $\widehat\varkappa\colon\mathrm{SL}_2(\C)\to\mathcal{C}$.
The fiber of $\widehat\varkappa$ over $(O,O)$, the tip of the cone, is $\mathrm{SU}(2)$,
and for all other points $(P_1,P_2)$ the fiber is a circle.
The images of our tropicalizations through $\widehat\varkappa$ are called
\textit{double-hyperbolic tropicalizations}. They dominate the hyperbolic tropicalizations.
We write $\mathrm{pr}_1:\mathrm{Cone}_\R(Q_2)\to\mathrm{SL}_2(\C)/\mathrm{SU}(2)$ for the
first projection, \ie $\varkappa=\mathrm{pr}_1\circ\widehat\varkappa$, and
$\pi_1:\mathrm{Cone}_\R(Q_2(\C))\to\mathrm{Cone}_\R(\C\Pj^1)$ induced by the
projection on the first factor of $\Q_2(\C)=\C\Pj^1\times\C\Pj^1$.
In \cref{fig:diagram}, we represent the different levels of tropicalizations and mappings.
Each level dominates the level below it.

\begin{figure}[ht]
\centering
\begin{tikzcd}
\parbox{3cm}{Phase\\ Space}&[-1.4cm]
    \R\times\C^4 &[-1.5cm]& \K^4\setminus\{0\} \ar[ll,"\Init_\nu"'] \\
\parbox{3cm}{Valuative\\ Tropicalization} &
    \big(\{0\}\times\mathrm{SL}_2(\C)\big)\cup
    \big((0,+\infty)\times\{\det=0\}\big)
    \ar[d,"\pi_\R"'] \arrow[draw=none]{u}[sloped,auto=false]{\subseteq} &&
        \mathrm{SL}_2(\K) \ar[ll] \ar[d,"\widetilde{\mathrm{VAL}}"]
        \arrow[draw=none]{u}[sloped,auto=false]{\subseteq} \\
\parbox{3cm}{$\mathrm{SL}_2$ \\ Tropicalization} &
    \big(\{0\}\times\mathrm{SU}(2)\big)\cup
    \big((0,+\infty)\times\mathcal{S}\big)
    \ar[rr,"\Xi"] \ar[d,"\pi_\C"'] &&
        \mathrm{SL}_2(\C) \ar[d,"\widehat{\varkappa}"] \\
\parbox{3cm}{Double-hyperbolic\\ Tropicalization} &
    \{*\}\cup\big((0,+\infty)\times Q_2(\C)\big) \ar[d,"\pi_1"'] \ar[rr,equal] &&
        \mathrm{Cone}_\R(Q_2(\C)) \ar[d,"\mathrm{pr}_1"] \\
\parbox{3cm}{Hyperbolic\\ Tropicalization} &
    \mathrm{Cone}_\R(\C\mathbb{P}^1) \arrow[draw=none]{r}[sloped,auto=false]{\equiv} &
        \mathbb{H}^3 \arrow[draw=none]{r}[sloped,auto=false]{\equiv} &
            \mathrm{SL}_2(\C)/\mathrm{SU}(2)
\end{tikzcd}
\caption{The different levels of tropicalization}
\label{fig:diagram}
\end{figure}
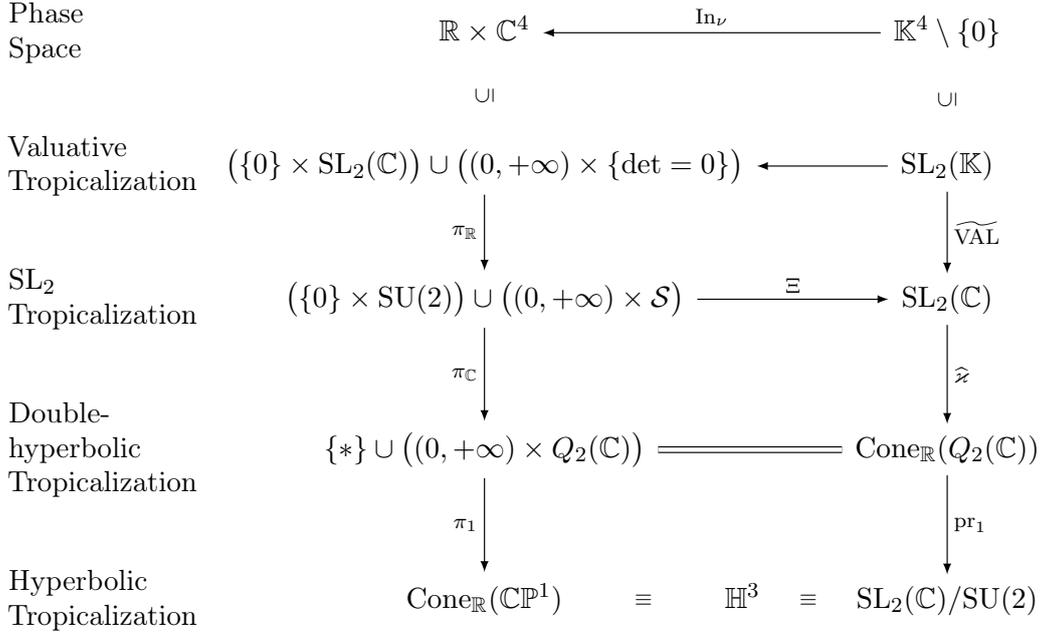

In order to prove \cref{thm:surfaces}, we describe the picture at the level of the valuative
tropicalization (second row in \cref{fig:diagram}). For any surface $X\subset\mathrm{SL}_2(\K)$
\[\Init_\nu(X) = \bigcup_{k=0}^r\{\beta_i\}\times\tilde S_i
\cup\bigcup_{k=0}^r(\beta_i,\beta_{i+1})\times \tilde C_i,\]
where $\tilde C_i$ is a homogeneous surface projecting to the above curve $C_i$ under
$\{\det=0\}\rightarrow Q_2$ and $\tilde S_i$ is an inhomogeneous surface in the total
space of the $\mathbb{C}^*$-bundle $\{\det=0\}$ over the quadric $Q_2$, which is a
particular case of \cref{thm:phase_trop} and \cref{prop:crit_values}. Even more concretely,
there is a sequence of homogeneous polynomials $f_0,f_1,\dots f_{d_{r+1}}$ of degrees
$0,1,\dots d_{r+1},$ such that $\tilde C_i$ is defined by $f_{d_{i+1}}=0$ and $S_i$
is given by $f_{d_{i-1}}+f_{d_{i-1}+1}+\cdots+f_{d_{i}}=0$. Moreover, every set of the above
form is realizable as a tropicalization of a surface given by $\sum t^{\gamma_i}f_i=0$,
such that the tropical polynomial $T\mapsto\max_i(\gamma_i+iT)$ has roots
$\{\beta_i:\:i=0,\ldots,r\}$. This is essentially the content of \cref{prop:crit_values} and
\cref{prop:principal}.

\begin{rke}\label{rke:future_facts}
We enumerate several facts about the fibers $X_P:=\varkappa^{-1}(P),\:P\in\mathbb{H}^3$ for
a general surface $X$.
\begin{enumerate}
\item First of all, the fibers $X_P$ are real semi-algebraic sets, thus we may speak about
the dimension of their irreducible components. One striking fact is that zero-dimensional
fibers never appear.
\item For a generic point $P$ the fiber $X_P$ is a generalized Hopf link, \ie a union of
pairwise linked circles in $\varkappa^{-1}(P)\cong\mathbb{S}^3$. More specifically, $X_P$
is a preimage of a finite set under a Hopf fibration defined by taking the quotient by
the stabilizer of $P$ under the action by $\mathrm{SU}(2)$, \ie
$\mathbb{S}^3=\mathrm{SU}(2)\to\mathrm{SU}(2)/\mathrm{Stab}(P)\cong\mathbb{S}^2$,
this finite subset does not locally depend on the level $\mathrm{dist}(O,P)$ while it is non-critical.
\item There is a finite number of {\it critical levels}, the first of which is $r>0$,
the radius of the open ball
$B_r(O)=\mathbb{H}^3\setminus\varkappa(\widetilde{\mathrm{VAL}}(X))$.
For each critical level $\beta>0$, a generic point $P\in\mathbb{H}^3$ at distance $\beta$
from $O$ has a two-dimensional fiber $X_P$. More precisely, $X_P\subset\mathbb{S}^3\subset\C^2$
may be seen as the radial projection of a complex algebraic curve in $\C^2$ with Newton
polygon having vertices $(n,0),(m,0),(0,n),(0,m)$, for some natural $n$ and $m$
depending only on the critical level.
\item For some phase tropicalizations $X$ there might be also three-dimensional fibers $X_P$.
This can happen, for instance, if the first critical level $r$ is $0$, \ie
the phase-forgetful tropicalization is the whole $\mathbb{H}^3$. Then $X_O$ is a spherical
coamoeba of some complex surface in $\mathrm{SL}_2(\C)$, which is often, but not always,
three-dimensional. Otherwise, the locus of $P$ with $X_P$ three-dimensional (and actually being the whole three-sphere) is a
finite union of segments extending to geodesics passing through $O$, which corresponds to the case when the above curve $C_i$ has an irreducible component of bi-degree $(0,1).$  
\end{enumerate}
\end{rke}

Similarly, tropicalization of curves in $\mathrm{SL}_2(\K)$ admit a floor diagram
decomposition, for which the floors are the critical levels. This is a phase extension of
the analogous fact about the geometric statement for tropical limits of hyperbolic
amoebas of families of curves established earlier in \cite{MS22} using very different
techniques.

To complete the picture, we clarify that if $\beta_0=0$, \ie the hyperbolic tropicalization
is the whole $\mathbb{H}^3$, the $\mathrm{SL}_2$ tropicalization acquires an extra
spherical coamoeba component $\{0\}\times\varkappa^\circ(S_0)$, for some surface $S_0$
in $\mathrm{SL}_2(\C)$, where $\varkappa^\circ:\mathrm{SL}_2(\C)\to\mathrm{SU}(2)$ is the
coamoeba map defined by taking the unitary part in the polar decomposition. At the
level of valuative tropicalization, this contribution simply becomes $\{0\}\times S_0$. 

At the more conceptual level, the case of $\mathrm{SL}_2$ may be seen as another building
block, besides the well-understood tori. The class of connected algebraic reductive groups
appears to be the suitable basis for performing tropicalization, since they possess a compact
real form, which would provide the phase for the generalized amoeba maps.

We finally formulate a couple of conjectural statements in line with this paper.
By Klein's theorem (that any hypersurface on a smooth quadric is a hypersurface section),
$X=\mathbb{V}(f,\det-1)$.
\begin{enumerate}
\item Assume the following.
\begin{enumerate}
    \item $\Init_\nu(X)$ is given by this additional equation $f$, \ie
    \[\Init_\nu(X)=\bigsqcup_{\alpha\in\R}\mathbb{V}(\IN_{\underline{\alpha}}(f),
    \IN_{\underline{\alpha}}(\det-1)).\]
    This situation is general in the coefficients of $f$.
    \item The polynomial $f$ has exactly $d$ tropical roots, where $d=\deg(f)$, they are
    positive and coincide with the critical levels of $X$.
    \item The curves from \cref{thm:surfaces} $C_1,\ldots,C_d$ are smooth and $C_i,\:C_{i+1}$
    intersect transversally. By Bertini's theorems, this is again a general situation.
\end{enumerate}
We can then conclude that the closure of $\mathrm{SL}_2(\mathbb{C})$ phase tropical diagram is a
topological $4$ manifold\footnote{Communicated to us by Ilia Zharkov.}. It is expected that it
is homeomorphic to the smooth complex surface given by a generic degree $d$ complex polynomial
and $\det-1$.
\item In the above assumptions, assume in addition that the $\mathrm{SL}_2$ phase tropical diagram is
invariant under conjugation of coordinates. We conjecture that its real part realizes an isotopy
type of a smooth real algebraic surface in $\mathrm{SL}_2(\mathbb{R})$ of degree $2d$. This justifies future
work towards an $\mathrm{SL}_2(\R)$ version of Viro's patchworking.
\end{enumerate}

Finally we wish to indicate the limit of our process, hinting at the difficulties towards
generalizing our work to higher dimensions or other reductive groups.

\begin{ex}\label{ex:counter-ex-dim3}
These results do not immediately generalize to $\mathrm{SL}_\ell$ for $\ell\geqslant3$.
Fix a value $\alpha\in(0,1)$ and consider, for instance, the diagonal matrix
\[A_t=\begin{pmatrix} t & & \\  & t^\alpha & \\ & & t^{-1-\alpha}\end{pmatrix}
=t\begin{pmatrix} 1 && \\ &0& \\ &&0\end{pmatrix}+o(t).\]
Clearly
\[\widetilde{\mathrm{VAL}}(A)=\begin{pmatrix} e && \\ & e^\alpha & \\ && e^{-1-\alpha}
\end{pmatrix}\]
depends on $\alpha$, however the first asymptotic of $A_t$ does not recover $\alpha$.
One needs to account for the lower order asymptotics.
\end{ex}

\section{Preliminaries}\label{sec:prelim}
In this section we present the basics of valuation theory necessary to define
a phase space and prove a generalized version of Kapranov's theorem.
We adopt the convention that the set $\N$ contains
$0$: $\N=\{0,1,2,3,\ldots\}$. We use the $(\max,+)$ convention for valuations,
\ie we adopt the following notations. Let $R$ be a commutative, unitary ring and
$\Gamma$ a totally ordered abelian group. We write $+$ for the addition law
of this group and we add an element not in $\Gamma$, that we denote $-\infty$.
We write $\Gamma_\infty=\Gamma\cup\{-\infty\}$ and extend the addition operation
and order relation so that $-\infty$ plays the role of the smallest element.

\begin{defn}
A \emph{valuation} $\nu$ on $R$ is a map $\nu:R\longrightarrow\Gamma_\infty$ satisfying
\begin{itemize}[noitemsep]
    \item[(V1)] $\forall a,b\in R,\:\nu(a\cdot b)=\nu(a)+\nu(b)$.
    \item[(V2)] $\forall a,b\in R,\:\nu(a+b)\leqslant\max\{\nu(a),\nu(b)\}$. This
    is called the ultrametric inequality.
    \item[(V3)] $\nu(1_R)=0$ and $\nu(0)=-\infty$.
    \item[(V4)] $\nu^{-1}(-\infty)=(0)$.
\end{itemize}
We often write $(R,\nu)$ for a \emph{valuative pair}, \ie a ring equipped with a valuation.
The mapping $\nu:R\setminus\{0\}\to\Gamma$ may not be surjective.
We write $\Gamma_\nu$ the abelian group generated by $\nu(R\setminus\{0\})$, called
the \textit{group of values}. The valuation $\nu$ is called \textit{trivial} if
$\Gamma_\nu=\{0\}$. We only consider nontrivial valuations in this paper.
\end{defn}

\begin{rke}\label{rke:basics}
\begin{enumerate}
\item When we talk of embedding a pair $(R,\nu)$ into $(S,\mu)$, we
mean we set an injective morphism $R\xhookrightarrow{\iota} S$, such
that $\mu\circ\iota=\nu$. If $R\subset S$ we can abbreviate by saying that
$\mu$ is an extension of $\nu$.
\item (V4) implies that $R$ is a domain. One could then extend $\nu$
to a valuation over $\text{Frac}(R)$, the quotient field of $R$, by setting
$\nu(a/b)=\nu(a)-\nu(b)$. It is easy to see it is well-defined.
\item For any integer $n\in\Z,\:\nu(n\cdot1_R)\leqslant0$. Indeed
\[\nu(n\cdot1_R)=\nu(\underbrace{1_R+\ldots+1_R}_{n\text{ times}})\leqslant\nu(1_R)=0.\]
Furthermore, we have $\nu(-1_R)=\nu(1_R)=0$ as $0=\nu(1_R)=\nu((-1_R)^2)=2\nu(-1_R)$.
\item By (V2) one can prove the following: for any $a,b\in R$, if $\nu(a)<\nu(b)$, then
$\nu(a+b)=\nu(a)$.
\end{enumerate}
\end{rke}

For the special case when $R$ is a field we define the following.

\begin{defn}
The \emph{valuation ring} of the valued field $(\K,\nu)$ is a local subring of $\K$,
written $\Or_\nu$ with maximal ideal $\m_\nu$ and residual field $\kappa_\nu$.
We define them as follows:
\begin{align*}
\Or_\nu &= \{a\in K\ |\ \nu(a)\leqslant0\}\\
\m_\nu &= \{a\in K\ |\ \nu(a)<0\}\\
\kappa_\nu &= \Or_\nu/\m_\nu.
\end{align*}
We will write $\bar{a}\in\kappa_\nu^\times$ for $\init_\nu(a)$, where $a\in\K$ is such that
$\nu(a)=0$ (\ie $a\in\Or_\nu^\times$).
\end{defn}

From a valuative pair $(R,\nu)$ we construct the \textit{graded algebra} of the pair.
For $\gamma\in\Gamma$ define the following groups

\begin{align*}
\mathcal{I}_\gamma=\mathcal{I}_\gamma(R,\nu) &= \{a\in R\ |\ \nu(a)\leqslant\gamma\}\\
\mathcal{I}^+_\gamma=\mathcal{I}^+_\gamma(R,\nu) &= \{a\in R\ |\ \nu(a)<\gamma\}.
\end{align*}

The graded ring $\gr_\nu(R)$ is
\[\gr_\nu(R)=\bigoplus_{\gamma\in\Gamma}\frac{\mathcal{I}_\gamma}{\mathcal{I}^+_\gamma}.\]

We abbreviate the $\gamma$ homogeneous factor of $\gr_\nu(R)$ by
$g_\gamma=\mathcal{I_\gamma}/\mathcal{I}^+_\gamma$.
The graded ring comes equipped with a map, called the \emph{initial form},
$\init_\nu\ :\ R\setminus\{0\}\to\text{gr}_\nu(R)$ assigning to $a\in R$
its class modulo $\mathcal{I}_\gamma^+(R,\nu)$ with $\gamma=\nu(a)$. We can additionally
extend the initial form map, by assigning $\init_\nu(0)=0$.

By definition, any homogeneous element \footnote{\ie any element belonging to one of the
direct factors $\mathcal{I}_\gamma/\mathcal{I}^+_\gamma$.}
is thus the initial form of some element in $R$ and any initial form
of any non-zero element is a non-zero element in the graded algebra.

\begin{rke}
\begin{enumerate}
\item IF $R=\K$ is a field, then $\kappa_\nu$ is simply $g_0$, the $0\in\Gamma_\nu$
component of $\gr_\nu(\K)$.
\item The valuation $\nu$ can be factored through the graded algebra in the following way.
We call the degree of a homogeneous element $\theta=\init_\nu(a)\in\gr_\nu(R)$ the value
$\gamma\in\Gamma$ such that $\theta\in\mathcal{I}_\gamma/\mathcal{I}^+_\gamma$.
Clearly $\gamma=\nu(a)$. For any element $\theta=\sum_i\theta_i\in\gr_\nu(R)$ with
$\theta_i$ homogeneous, we call $\deg(\theta)$, the smallest of $\deg(\theta_i)$.
We thus have a factorization
\[\begin{tikzcd}
& \gr_\nu(R) \ar[rd,"\deg"] & \\
R\setminus\{0\} \ar[ur,"\init_\nu"] \ar[rr,"\nu"] & & \Gamma
\end{tikzcd}\]
\item If $R=\K$ is a field, then each $g_\gamma$ is a natural $\kappa_\nu$ vector
space: for any $a\in\mathcal{O}_\nu$ and $r\in\K$ we have
$\overline{a}\cdot\init_\nu(r)=\init_\nu(ar)$.
\end{enumerate}
\end{rke}

For general filtered modules, rings, algebras etc. there is also a
notion of initial form which fails to be a morphism in general. It may not
even be multiplicative, however when considering the graded algebra
associated to a valuation, we can state some simple properties.

\begin{prop}\label{prop:in_compute}
Let $a,b\in R$. We have the following
\begin{enumerate}[noitemsep]
\item $\init_\nu(a\cdot b)=\init_\nu(a)\cdot\init_\nu(b)$.
Thus the graded ring is an integral domain.
\item if $\nu(a)>\nu(b)$ then $\init_\nu(a+b)=\init_\nu(a)$.
\item if $\nu(a)=\nu(b)>\nu(a+b)$ then $\init_\nu(a+b)\neq\init_\nu(a)+\init_\nu(b)=0$.
\item if $\nu(a)=\nu(b)=\nu(a+b)$ then $\init_\nu(a+b)=\init_\nu(a)+\init_\nu(b)$.
\item we can generalize the previous point to any finite number of terms: for all
$x_1,\ldots,x_n\in R$ with $\nu(x_1)=\ldots=\nu(x_n)=\nu(x_1+\cdots+x_n)$ we have
\[\init_\nu(x_1+\cdots+x_n)=\init_\nu(x_1)+\cdots+\init_\nu(x_n).\]
\end{enumerate}
\end{prop}

\begin{proof}
Statement (1) is simply a consequence of (V1). The rest amount to using (V2) and the fact
that $\mathcal{I}_\gamma\to\mathcal{I}_\gamma/\mathcal{I}^+_\gamma$ is a group morphism.
\end{proof}

\begin{ex}\label{ex:Hahn}
Fix any field $F$ and an abelian ordered group $\Gamma$. We define the valued field
of Hahn series $\K=F[[t^{\Gamma}]]$ as follows. A Hahn series is a formal expression of
the form $x=\sum_{\gamma\in\Gamma}x_\gamma t^\gamma,\:x_\gamma\in F$, where the set
$-\mathrm{supp}(x):=\{-\gamma:\:x_\gamma\neq0\}$ is well-ordered: any subset of
$\mathrm{supp}(x)$ has a largest element. One can define the term-wise addition and the
convolution product, thus giving $F[[t^{\Gamma}]]$ a field structure (\cf \cite{Poo} for
details). Finally we equip this field with a canonical valuation $\nu=\ord_t$, defined as
\[\ord_t(x):=\max(\mathrm{supp}(x)).\]
Let $\Gamma_{\leqslant0}$ be the set of non-negative elements of $\Gamma$ and $\Gamma_{<0}$
the set of its negative elements. We have
\[\Or_{\nu}=F[[t^{\Gamma_{\leqslant0}}]],\:\mathfrak{m}_\nu=F[[t^{\Gamma_{<0}}]],
\:\kappa_\nu=F,\:\text{ and }\:\gr_\nu(\K)=F[\Gamma],\]
where $F[\Gamma]$ is the group algebra of $\Gamma$ over $F$, \ie it is the set of formal
expressions $\sum_{\gamma\in\Gamma}x_\gamma[\gamma]$ where all but finitely many
$x_\gamma\in F$ are zero. To focus on the dominating term of
$x=\sum_{\gamma\in\Gamma}x_\gamma t^\gamma$, \ie the one of highest value, we often write
$x=x_\alpha t^\alpha+o(t^\alpha)$ where $\alpha=\ord_t(\alpha)$.

In fact for any ring $R$ such that $F[t^\Gamma]\subset R\subset F[[t^{\Gamma}]]$, equipped with
the restriction of $\ord_t$ as its valuation $\nu$, we have $\gr_\nu(R)=F[t^\Gamma]$.
\end{ex}

\section{Composing initial forms}\label{sec:composing_init}

A key ingredient in proving our Kapranov theorem is functoriality. One could see this
result as a means to substantiate a change of variables on the level of the graded
algebra, mirroring a change of variables on the level of the polynomial algebra.

\subsection{Functoriality}
We show how the construction of the graded algebra is functorial.
We extend \cite[Thm. 1.8, p. 9]{BL}.

\begin{thm}\label{thm:funct}
Consider two valuative pairs $(R_1,\nu_1),\:(R_2,\nu_2)$. We assume $\nu_1,\nu_2$
take values inside a common ordered subgroup $\Gamma$. Let $\Phi:R_1\to R_2$
be a ring morphism such that $\nu_1\geqslant\nu_2\circ\Phi$,
\ie $\forall a\in R_1,\:\nu_1(a)\geqslant\nu_2(\Phi(a))$.
\begin{enumerate}
\item There is a canonical homogeneous map  of graded rings
(\ie sending homogeneous elements to either $0$ or homogeneous of same degree)
\[\gr(\Phi):\gr_{\nu_1}(R_1)\longrightarrow\gr_{\nu_2}(R_2),\]
sending $\init_\nu(a),\:a\in R_1$ to $\Phi(a)\mod
\mathcal{I}^+_{\gamma}(R_2,\nu_2)$, where $\gamma=\nu_1(a)$.
\item The kernel of $\gr(\Phi)$ is the homogeneous ideal
\[\mathrm{Ker}(\gr(\Phi))=\Big\langle\init_{\nu_1}(I)\Big\rangle,
\text{ where }I=\{a\in R_1\,:\:\nu_1(a)>\nu_2(\Phi(a))\}.\]
\item The above construction $\gr(-)$ is functorial in the following sense.
Let $(R_3,\nu_3)$ be another valued pair and let $\Psi:R_2\to R_3$ be another
ring morphism such that $\nu_2\geqslant\nu_3\circ\Psi$. Then one has
\[\gr(\Psi\circ\Phi)=\gr(\Psi)\circ\gr(\Phi).\]
\end{enumerate}
\end{thm}

\begin{proof}
By hypothesis we have for every $\gamma\in\Gamma$
\begin{align*}
\Phi(\mathcal{I}_\gamma(R_1,\nu_1)) & \subseteq\mathcal{I}_\gamma(R_2,\nu_2)\\
\Phi(\mathcal{I}^+_\gamma(R_1,\nu_1)) & \subseteq\mathcal{I}^+_\gamma(R_2,\nu_2).
\end{align*}
This induces a mapping of abelian groups
\[\frac{\mathcal{I}_\gamma(R_1,\nu_1)}
{\mathcal{I}^+_\gamma(R_1,\nu_1)}\longrightarrow
\frac{\mathcal{I}_\gamma(R_2,\nu_2)}{\mathcal{I}^+_\gamma(R_2,\nu_2)},\quad
\init_\nu(a)\longmapsto\Phi(a)\mod\mathcal{I}^+_{\nu_1(a)}(R_2,\nu_2).\]
These maps put together generate a global mapping of abelian groups
$\gr(\Phi):\gr_{\nu_1}(R_1)\to\gr_{\nu_2}(R_2)$.

Since $\Phi$ is a ring morphism, it sends units to units, so $\gr(\Phi)$ clearly sends the
unit of $\gr_{\nu_1}(R_1)$ to the unit of $\gr_{\nu_2}(R_2)$. By (V1) and the fact that
$\Phi$ is multiplicative, $\gr(\Phi)$ is also multiplicative. It is enough to show it for two
homogeneous elements $\init_{\nu_1}(a),\init_{\nu_1}(b),\,a,b\in R_1$
\begin{align*}
\gr(\Phi)(\init_{\nu_1}(a)\cdot\init_{\nu_1}(b))
&=\gr(\Phi)(\init_{\nu_1}(ab))\\
&=\Phi(ab)\mod\mathcal{I}^+_{\nu_1(ab)}(R_2,\nu_2)\\
&=\Phi(a)\Phi(b)\mod\mathcal{I}^+_{\nu_1(a)+\nu_1(b)}(R_2,\nu_2)\\
&=\Big(\Phi(a)\mod\mathcal{I}^+_{\nu_1(a)}(R_2,\nu_2)\Big)\cdot
    \Big(\Phi(b)\mod\mathcal{I}^+_{\nu_1(b)}(R_2,\nu_2)\Big)\\
&=\gr(\Phi)(\init_{\nu_1}(a))\cdot\gr(\Phi)(\init_{\nu_1}(b)).
\end{align*}

The kernel of $\gr(\Phi)$ is the abelian subgroup of $\gr_{\nu_1}(R_1)$
generated by the homogeneous elements $\init_{\nu_1}(a)$ that are sent to $0$
in $\gr_{\nu_2}R_2$. In other words, these homogeneous elements are the
$\init_{\nu_1}(a)$ such that $\Phi(a)\in\mathcal{I}^+_{\nu_1(a)}(R_2,\nu_2)$.
This last condition is equivalent to $\nu_1(a)>\nu_2(\Phi(a))$.

Lastly, functoriality is a consequence of the fact that the morphism
$\frac{\mathcal{I}_\gamma(R_1,\nu_1)}{\mathcal{I}^+_\gamma(R_1,\nu_1)}\to
\frac{\mathcal{I}_\gamma(R_3,\nu_3)}{\mathcal{I}^+_\gamma(R_3,\nu_3)}$, induced by
$\Psi\circ\Phi$ is the composition of the morphisms
$\frac{\mathcal{I}_\gamma(R_2,\nu_2)}{\mathcal{I}^+_\gamma(R_2,\nu_2)}\to
\frac{\mathcal{I}_\gamma(R_3,\nu_3)}{\mathcal{I}^+_\gamma(R_3,\nu_3)}$ induced by $\Psi$
and
$\frac{\mathcal{I}_\gamma(R_1,\nu_1)}{\mathcal{I}^+_\gamma(R_1,\nu_1)}\to
\frac{\mathcal{I}_\gamma(R_2,\nu_2)}{\mathcal{I}^+_\gamma(R_2,\nu_2)}$ induced by $\Phi$.
\end{proof}

\subsection{Monomial valuations}
In order to give a better picture of phase tropicalizations, we will need to define a set
of valuations on $\K[x_1,\ldots,x_n]$ extending $\nu$. We frequently use multi-index notation:
if $u\in\N^n$, then $A^u$ is assumed to mean that the tuple $u=(u_1,\ldots,u_n)$ and the elements
$A_1,\ldots,A_n$ are well-defined, and $A^u=A_1^{u_1}\cdots A_n^{u_n}$.

These valuations on $\K[x_1,\ldots,x_n]$ can have an excessively complicated structure or
definition. We focus on a specific class of valuations that is easy to classify. Namely
we study the valuations which are perfectly determined by their values on monomials
(or equivalently, the values of the coordinates $x_i$).

We introduce some additional notations. Write $f=\sum_{u\in U}c_ux^u$. We define
\begin{align*}
\IN_{\underline{\gamma}}(f)
&=\sum_{u\in U_0}\init_\nu(c_u)X^u\in\gr_\nu(\K)[X_1,\ldots,X_n]\quad\text{and}\\
[f]_{\underline{\gamma}}&=\sum_{u\in U_0}c_ux^u\in\K[x_1,\ldots,x_n]
\end{align*}
where
\[U_0=\{u\in U:\:\nu_{\underline{\gamma}}(c_ux^u)=\nu_{\underline{\gamma}}(f)\}.\]

\begin{prop}\label{prop:mono_is_valuation}
Let $\nu:\K\to\Gamma_\infty$ be a valuation (with the mapping $\nu$ not necessarily
surjective). Let $\underline{\gamma}=(\gamma_1,\ldots,\gamma_n)\in\Gamma^n$ and for
$f=\sum_{u\in U}c_ux^u\in\K[x_1,\ldots,x_n]\setminus\{0\}$ define
\[\nu_{\underline{\gamma}}(f)=\max_{u\in U,\,c_u\neq0}\nu_{\underline{\gamma}}(c_ux^u)=
\max_{u\in U,\,c_u\neq0}\left(\nu(c_u)+\sum_{i=1}^nu_i\gamma_i\right).\]
The map $\nu_{\underline{\gamma}}$ is a valuation. We clearly have
$\nu_{\underline{\gamma}}(x^u)=u\cdot\underline{\gamma}:=\sum_{i=1}^nu_i\gamma_i$.
\end{prop}

\begin{proof}
Properties (V3) and (V4) are clear. Let us show (V2). Consider $f=\sum_{u\in U}c_ux^u$
and $g=\sum_{u\in U}d_ux^u$. Take $u^0\in U$ such that $\nu_{\underline{\gamma}}(f+g)=
\nu(c_{u^0}+d_{u^0})+u^0\cdot\underline{\gamma}$. We clearly have
\begin{align*}
\nu_{\underline{\gamma}}(f+g) &= \nu(c_{u^0}+d_{u^0})+u^0\cdot\underline{\gamma} \\
&\leqslant\max\{\nu(c_{u^0}),\nu(d_{u^0})\}+u^0\cdot\underline{\gamma} \\
&=\max\left\{\nu(c_{u^0})+u^0\cdot\underline{\gamma},\:
    \nu(d_{u^0})+u^0\cdot\underline{\gamma}\right\} \\
&\leqslant\max\{\nu_{\underline{\gamma}}(f),\:\nu_{\underline{\gamma}}(g)\}.
\end{align*}
It remains to show (V1) which is the most difficult property to prove.
Fix $f=\sum_{u\in U}c_ux^u$ and $g=\sum_{v\in V}d_vx^v$. We proceed in several steps,
by showing a double inequality between $\nu_{\underline{\gamma}}(fg)$ and
$\nu_{\underline{\gamma}}(f)+\nu_{\underline{\gamma}}(g)$. Note that since we have (V2)
we can use one of its consequences, namely \cref{rke:basics} (4).

\textit{First step:} we show one inequality. For arbitrary $f,g$ write
\[fg=\sum_{w\in W}e_wx^w,\quad\mathrm{where}\quad
e_w=\sum_{\substack{u\in U,\:v\in V \\u+v=w}}c_ud_v.\]
Then
\begin{align*}
\nu_{\underline{\gamma}}(fg) &
    \leqslant\max_{w\in W}\max_{\substack{u\in U,\:v\in V \\u+v=w}}
    \nu_{\underline{\gamma}}(c_ux^u)+\nu_{\underline{\gamma}}(d_vx^v) \\
&=\max_{u\in U,v\in V}\nu_{\underline{\gamma}}(c_ux^u)+\nu_{\underline{\gamma}}(d_vx^v) \\
&=\max_{u\in U}\nu_{\underline{\gamma}}(c_ux^u)+\max_{v\in V}\nu_{\underline{\gamma}}(d_vx^v) \\
&=\nu_{\underline{\gamma}}(f)+\nu_{\underline{\gamma}}(g).
\end{align*}

\textit{Second step:}
Consider a monomial ordering on $\K[x_1,\ldots,x_n]$ (\eg the lexicographic ordering).
Consider $u_0\in U_0$ and $v_0\in V_0$ such that $c_{u_0}x^{u_0}$ and $d_{v_0}x^{v_0}$ are
the highest terms in $[f]_{\underline\gamma}$ and $[g]_{\underline\gamma}$ respectively,
and set $w_0=u_0+v_0$. Then $e_{w_0}x^{w_0}$ is the highest term in
$[f]_{\underline\gamma}[g]_{\underline\gamma}$ and it can be decomposed
in $e_{w_0}x^{w_0}=c_{u_0}x^{u_0}d_{v_0}x^{v_0}$. We thus have
\begin{align*}
\nu_{\underline{\gamma}}([f]_{\underline\gamma})+
\nu_{\underline{\gamma}}([g]_{\underline\gamma})&\geqslant
\nu_{\underline{\gamma}}([f]_{\underline\gamma}[g]_{\underline\gamma})\geqslant
\nu_{\underline{\gamma}}(e_{w_0}x^{w_0})
=\nu_{\underline{\gamma}}(c_{u_0}x^{u_0}d_{v_0}x^{v_0}) \\
&=\nu_{\underline{\gamma}}(c_{u_0}x^{u_0})+\nu_{\underline{\gamma}}(d_{v_0}x^{v_0})
=\nu_{\underline{\gamma}}([f]_{\underline\gamma})+
\nu_{\underline{\gamma}}([g]_{\underline\gamma}).
\end{align*}
In conclusion, the above inequalities are in fact equalities, thus
\[\nu_{\underline{\gamma}}([f]_{\underline\gamma}[g]_{\underline\gamma})=
\nu_{\underline{\gamma}}([f]_{\underline\gamma})+
\nu_{\underline{\gamma}}([g]_{\underline\gamma}).\]

\textit{Third step:} we show that $\nu_{\underline{\gamma}}(fg)=
\nu_{\underline{\gamma}}([f]_{\underline\gamma}[g]_{\underline\gamma})$.
Observe that
\[\nu_{\underline{\gamma}}(f-[f]_{\underline\gamma})<
\nu_{\underline{\gamma}}(f)=\nu_{\underline{\gamma}}([f]_{\underline\gamma})\quad\text{and}\quad
\nu_{\underline{\gamma}}(g-[g]_{\underline\gamma})<
\nu_{\underline{\gamma}}(g)=\nu_{\underline{\gamma}}([g]_{\underline\gamma}).\]
By the first step we have
\begin{align*}
\nu_{\underline{\gamma}}(fg-[f]_{\underline\gamma}[g]_{\underline\gamma}) &= 
\nu_{\underline{\gamma}}(f(g-[g]_{\underline\gamma})+
[g]_{\underline\gamma}(f-[f]_{\underline\gamma})) \\
&\leqslant\max\left\{\nu_{\underline{\gamma}}(f(g-[g]_{\underline\gamma})),
\nu_{\underline{\gamma}}([g]_{\underline\gamma}(f-[f]_{\underline\gamma}))\right\} \\
&\leqslant\max\left\{\nu_{\underline{\gamma}}(f)+
\nu_{\underline{\gamma}}(g-[g]_{\underline\gamma}),
\nu_{\underline{\gamma}}([g]_{\underline\gamma})+
\nu_{\underline{\gamma}}(f-[f]_{\underline\gamma})\right\} \\
&<\nu_{\underline{\gamma}}([f]_{\underline\gamma})+
\nu_{\underline{\gamma}}([g]_{\underline\gamma}) \\
&=\nu_{\underline{\gamma}}([f]_{\underline\gamma}[g]_{\underline\gamma}).
\end{align*}

By using, in succession, steps three, two and one, we obtain
\[\nu_{\underline{\gamma}}(fg)
=\nu_{\underline{\gamma}}([f]_{\underline\gamma}[g]_{\underline\gamma})
=\nu_{\underline{\gamma}}([f]_{\underline\gamma})+
\nu_{\underline{\gamma}}([g]_{\underline\gamma})
=\nu_{\underline{\gamma}}(f)+\nu_{\underline{\gamma}}(g),\]
which concludes our proof.
\end{proof}

We grade the $\gr_\nu(\K)[X_1,\ldots,X_n]$ with respect to $\nu_{\underline{\gamma}}$:
any element of $\gr_\nu(\K)[X_1,\ldots,X_n]$ can be uniquely decomposed as sum of terms
$\theta X^u,\:u\in\N^n$ with $\theta\in\gr_\nu(\K)$ homogeneous. We associate to it the degree
$\deg(\theta)+u\cdot\underline{\gamma}$. It should come as no surprise that
$\deg(\init_\nu(c)X^u)=\nu_{\underline{\gamma}}(cx^u)$.

Observe finally that $\gr_{\nu_{\underline{\gamma}}}(\K[x_1,\ldots,x_n])$ is a natural
$\gr_\nu(\K)$-algebra. Indeed, by functoriality \cref{thm:funct}, there is a clear
injective ring map $\gr_\nu(\K)\to\gr_{\nu_{\underline{\gamma}}}(\K[x_1,\ldots,x_n])$.
Thus we define the algebra structure as follows: for any $c\in\K,f\in\K[x_1,\ldots,x_n]$
$\init_\nu(c)\cdot\init_{\nu_{\underline{\gamma}}}(f)=\init_{\nu_{\underline{\gamma}}}(cf)$.

Monomial valuations have graded algebras that can be described explicitly as is
demonstrated by the next proposition.

\begin{prop}\label{prop:graded_monomial}
For a monomial valuation $\nu_{\underline{\gamma}}$, the graded ring
$\gr_{\nu_{\underline{\gamma}}}(\K[x_1,\ldots,x_n])$ is isomorphic
(as graded rings) to $\gr_\nu(\K)[X_1,\ldots,X_n]$ (with $X_i$ of weight $\gamma_i$).
Namely
\[\gr_{\nu_{\underline{\gamma}}}(\K[x_1,\ldots,x_n])=\gr_\nu(\K)
[\init_{\nu_{\underline{\gamma}}}(x_1),\ldots,\init_{\nu_{\underline{\gamma}}}(x_n)],\]
where the $\init_{\nu_{\underline{\gamma}}}(x_i)$ are algebraically independent over
$\gr_\nu(\K)$. More precisely, the initial form of $\init_{\nu_{\underline{\gamma}}}(f)$
can be written as the polynomial
$\IN_{\underline{\gamma}}(f)(\init_\nu(x_1),\ldots,\init_\nu(x_n))$.
\end{prop}

\begin{proof}
We fix $f=\sum_{u\in U}c_ux^u\in\K[x_1,\ldots,x_n]$. By definition
$\nu_{\underline{\gamma}}(f-[f]_{\underline{\gamma}})<\nu_{\underline{\gamma}}(f)=
\nu_{\underline{\gamma}}([f]_{\underline{\gamma}})$. Thus
\begin{equation}\label{eq:reduc_1}
\init_{\nu_{\underline{\gamma}}}(f)=
\init_{\nu_{\underline{\gamma}}}([f]_{\underline{\gamma}}).
\end{equation}
Furthermore, each term $c_ux^u$ of $[f]_{\underline{\gamma}}$ have equal valuation to that of
$[f]_{\underline{\gamma}}$. By \cref{prop:in_compute} (5) we have
\begin{equation}\label{eq:reduc_2}
\init_{\nu_{\underline{\gamma}}}([f]_{\underline{\gamma}})
=\sum_{u\in U_0}\init_{\nu_{\underline{\gamma}}}(c_ux^u)
=\sum_{u\in U_0}\init_\nu(c_u)\init_{\nu_{\underline{\gamma}}}(x)^u
=\IN_{\underline{\gamma}}(f)(\init_{\nu_{\underline{\gamma}}}(x)).
\end{equation}
By combining \eqref{eq:reduc_1} and \eqref{eq:reduc_2} we conclude that indeed
$\gr_{\nu_{\underline{\gamma}}}(\K[x_1,\ldots,x_n])$ is a polynomial algebra over
$\gr_\nu(\K)$. We show that the elements $\init_{\nu_{\underline{\gamma}}}(x_1),\ldots,
\init_{\nu_{\underline{\gamma}}}(x_n)$ are algebraically independent over $\gr_\nu(\K)$.
Let $F\in\gr_\nu(\K)[X_1,\ldots,X_n]$ be non-zero. We can decompose it into homogeneous
components $F=F_1+\ldots+F_N,\:F_i\neq0$, according to the grading of
$\gr_\nu(\K)[X_1,\ldots,X_n]$. For each $i=1,\ldots,N$ it is clear that
$F_i(\init_{\nu_{\underline{\gamma}}}(x))$ is homogeneous of degree $\deg(F_i)$.
Without loss of generality, we can thus suppose that $F$ is homogeneous and non-zero
and we show that $F(\init_{\nu_{\underline{\gamma}}}(x))\neq0$. Indeed, since $F$ is
homogeneous, we can write $F=\sum_{u\in U}\init_{\nu}(c_u)X^u$. Since $F$ is homogeneous,
we clearly have $F=\IN_{\underline{\gamma}}(f)$, where $f=\sum_{u\in U}c_ux^u$. Thus, by a
computation that is similar to that in the first part of the proof
\[F(\init_{\nu_{\underline{\gamma}}}(x))
=\IN_{\underline{\gamma}}(f)(\init_{\nu_{\underline{\gamma}}}(x))
=\init_{\nu_{\underline{\gamma}}}(f)\]
and $\init_{\nu_{\underline{\gamma}}}(f)\neq0$.
\end{proof}

\begin{cor}\label{cor:IN_mult}
We deduce that $\IN_{\underline{\gamma}}$ is multiplicative: for any $f,g\in\K[x_1,\ldots,x_n]$
\[\IN_{\underline{\gamma}}(fg)=\IN_{\underline{\gamma}}(f)\IN_{\underline{\gamma}}(g).\]
\end{cor}

\begin{proof}
We apply \cref{prop:graded_monomial} and use the fact that $\init_\nu$ is multiplicative:
\begin{align*}
\IN_{\underline{\gamma}}(fg)(\init_\nu(x_1),\ldots,\init_\nu(x_n))
&=\init_{\nu_{\underline{\gamma}}}(fg)\\
&=\init_{\nu_{\underline{\gamma}}}(f)\cdot\init_{\nu_{\underline{\gamma}}}(g) \\
&=\IN_{\underline{\gamma}}(f)(\init_\nu(x_1),\ldots,\init_\nu(x_n))\\
&\phantom{=}\cdot\IN_{\underline{\gamma}}(g)(\init_\nu(x_1),\ldots,\init_\nu(x_n))
\end{align*}
and we conclude as the elements $\init_\nu(x_1),\ldots,\init_\nu(x_n)$ are algebraically
independent.
\end{proof}

\subsection{Changing variables for initial forms}

Throughout this subsection, $(\K,\nu)$ is a fixed valued field.

\begin{lm}\label{lm:substitute}
Consider a morphism of polynomial algebras $\Psi:\K[x_1,\ldots,x_n]\to\K[y_1,\ldots,y_m]$
given by $\Psi(x_i)=g_i$ for $i=1,\ldots,n$. Equip $\K[x_1,\ldots,x_n]$ with a
monomial valuation $\nu_{\underline{\gamma}}$ and $\K[y_1,\ldots,y_m]$ with a monomial
valuation $\nu_{\underline{\delta}}$. We suppose that $\forall i=1,\ldots,n,\,
\nu_{\underline{\delta}}(g_i)\leqslant\gamma_i$.
If we identify $\gr_{\nu_1}(\K[x_1,\ldots,x_n])$ and $\gr_{\nu_2}(\K[y_1,\ldots,y_m])$
with their respective polynomial algebras $\gr_\nu(\K)[X_1,\ldots,X_n]$ and
$\gr_\nu(\K)[Y_1,\ldots,Y_m]$. Then $\gr(\Psi)$ is determined by the images of
$X_i=\init_{\nu_1}(x_i)$.
\begin{enumerate}
\item If $\nu_2(g_i)<\gamma_i$, then $\gr(\Psi)(X_i)=0$.
\item If $\nu_2(g_i)=\gamma_i$, then $\gr(\Psi)(X_i)=\init_{\nu_2}(g_i)$.
\end{enumerate}
\end{lm}

\begin{proof}
We show that $\Psi$ verifies the condition of \cref{thm:funct}. Take $f\in\K[x_1,\ldots,x_n]$
and write $f=\sum_{u\in U}c_ux^u$. By the ultrametric inequality we have
\begin{align*}
\nu_{\underline{\delta}}(\Psi(f))
&= \nu_{\underline{\delta}}\left(\sum_{u\in U}c_ug_1^{u_1}\cdots g_n^{u_n}\right) \\
&\leqslant\max_{u\in U}\nu_{\underline{\delta}}(c_ug_1^{u_1}\cdots g_n^{u_n}) \\
&=\max_{u\in U}\nu(c_u)+u_1\nu_{\underline{\delta}}(g_1)+\cdots+
    u_n\nu_{\underline{\delta}}(g_n) \\
&\leqslant\max_{u\in U}\nu(c_u)+u_1\gamma_1+\cdots+u_n\gamma_n \\
&=\max_{u\in U}\nu_{\underline{\gamma}}(c_ux^u)\\
&=\nu_{\underline{\gamma}}(f).
\end{align*}
Whether $\gr(\Psi)(X_i)=0$ or $\gr(\Psi)(X_i)=\init_{\nu_2}(g_i)$ comes from the point (2) of
\cref{thm:funct} and the construction of $\gr(\Psi)$.
\end{proof}

We are now ready to explain how to change variables for initial forms.

\begin{lm}\label{lm:sub}
Consider an isomorphism $\Phi:\K[x_1,\ldots,x_n]\to\K[y_1,\ldots,y_n]$. Fix
$\underline{\delta}=(\delta_1,\ldots,\delta_m)$ and
$\underline{\gamma}=(\gamma_1,\ldots,\gamma_n)$ such that $\forall i=1,\ldots,n$
\begin{align*}
    \nu_{\underline{\delta}}(\Phi(y_i)) &= \nu_{\underline{\gamma}}(y_i) \\
    \nu_{\underline{\gamma}}(\Phi^{-1}(x_i)) &= \nu_{\underline{\delta}}(x_i).
\end{align*}
Then $\gr(\Phi):\gr_{\nu_{\underline{\delta}}}(\K[x_1,\ldots,x_n])\to
\gr_{\nu_{\underline{\gamma}}}(\K[y_1,\ldots,y_n])$ is an isomorphism.

Upon identifying $\gr_{\nu_{\underline{\delta}}}(\K[y_1,\ldots,y_n])$ and
$\gr_{\nu_{\underline{\gamma}}}(\K[x_1,\ldots,x_n])$ with their respective
polynomial algebras $\gr_{\nu}(\K)[Y_1,\ldots,Y_n]$ and $\gr_{\nu}(\K)[X_1,\ldots,X_n]$,
one can see $\gr(\Phi)$ as an isomorphism between them. Thus
$\forall f\in\K[x_1,\ldots,x_n],\,\nu_{\underline{\delta}}(\Phi(f))=
\nu_{\underline{\gamma}}(f)$ and
\[\gr(\Phi)(\IN_{\underline{\gamma}}(f))=\IN_{\underline{\delta}}(\Phi(f)).\]
\end{lm}

Situations as the ones described above arise naturally in the following way: consider
$\gamma\in\Gamma$ and consider $\varphi_1,\ldots,\varphi_n$, independent linear forms
on $\K^n$ with $\varphi_i=\sum_ja_{i,j}x_j$ and $a_{i,j}=0$ or $\nu(a_{i,j})=0$.
Then the isomorphism defined by $\Phi(x_i)=\varphi(y_1,\ldots,y_n)$ verifies the
conditions of \cref{lm:sub}.

\begin{proof}
Both $\Phi$ and $\Phi^{-1}$ satisfy the conditions of \cref{lm:substitute}, thus both
maps $\gr(\Phi)$ and $\gr(\Phi^{-1})$ can be defined. By \cref{thm:funct} (3), we have
$\gr(\Phi)\circ\gr(\Phi^{-1})=\gr(\Phi\circ\Phi^{-1})=\gr(\mathrm{id})=\mathrm{id}$
and likewise $\gr(\Phi^{-1})\circ\gr(\Phi)=\mathrm{id}$. Thus $\gr(\Phi)$ is an
isomorphism. By \cref{thm:funct} (2), we see that there are no
$f\in\K[x_1,\ldots,x_n]$ such that $\nu_{\underline{\delta}}(\Phi(f))<
\nu_{\underline{\gamma}}(f)$. The last statements are a consequence of
\cref{prop:graded_monomial}.
\end{proof}

\section{Lifting phase tropicalizations}\label{sec:lifting}

\subsection{Phases}
For any natural number $n\in\N$, we set a norm on $\K^n$ as follows
\[z=(z_1,\ldots,z_n)\in\K^n,\:V_\nu(z)=\max_{1\leqslant i\leqslant n}\nu(z_i).\]
Just as we defined the graded algebra $\gr_\nu(\K)$ for the valued field, we can define the
graded abelian groups. For any $\alpha\in\Gamma_\nu$ set
\begin{align*}
    \mathcal{P}_\alpha(\K^n) &:= \{z\in\K^n:\,V_\nu(x)\leqslant\alpha\}\\
    \mathcal{P}^+_\alpha(\K^n) &:= \{z\in\K^n:\,V_\nu(x)<\alpha\}.
\end{align*}
and set
\[\Gr_\nu(\K^n):=\bigsqcup_{\alpha\in\Gamma_\nu}
\frac{\mathcal{P}_\alpha(\K^n)}{\mathcal{P}^+_\alpha(\K^n)}
\quad,\quad
\widetilde{\Gr}_\nu(\K^n):=\bigoplus_{\alpha\in\Gamma_\nu}
\frac{\mathcal{P}_\alpha(\K^n)}{\mathcal{P}^+_\alpha(\K^n)}.\]

There is a clear injection $\Gr_\nu(\K^n)\hookrightarrow\widetilde{\Gr}_\nu(\K^n)$.
If the context is clear, we abbreviate the $\alpha$ homogeneous component
as $G_\alpha:=\mathcal{P}_\alpha(\K^n)/\mathcal{P}^+_\alpha(\K^n)$. Just like $g_\alpha$,
the abelian group $G_\alpha$ is a $\kappa_\nu$ vector space.
We have an initial form mapping
\[\Init_\nu:\K^n\setminus\{0\}\to\Gr_\nu(\K^n),\:z\mapsto z\mod\mathcal{P}^+_\alpha(\K^n),
\text{ where }\alpha=V_\nu(x).\]

\begin{defn}
We call $\Gr_\nu(\K)$ the \textit{phase space} for $(\K,\nu)$.
We call its elements \textit{forms}. A form is thus a couple $(\alpha,\theta)$
where $\theta\in\kappa_\nu^n$. The associated degree $\alpha$ is the \textit{value}
of the form and $\theta$ is its \textit{phase}.

Let $X=\mathbb{V}(I)\subset\K^n$ be an affine algebraic set.
We define the \textit{phase tropicalization of} $X$ as
\[\Init_\nu(X):=\{\Init_\nu(z):\,z\in X\}\subset\Gr_\nu(\K^n).\]
This assignment is clearly functorial: any (set-theoretic) map $u:X\to Y$ gives
rise to a map $\Init_\nu(u):\Init_\nu(X)\to\Init_\nu(Y)$, the phase forms of degree $\alpha$.

We also set $\Init_\nu(X)_\alpha:=\Init_\nu(X)\cap G_\alpha$
\end{defn}

\begin{ex}
We have $\Init_\nu(\K^n\setminus\{0\})=\Gr_\nu(\K^n)\setminus(\Gamma_\nu\times\{0\})$.
\end{ex}

Similarly to the graded algebra construction, $\Gr_\nu$ is a functorial construction.

\begin{lm}\label{prop:funct_Gr}
Suppose $(\K,\nu)$ and $(\mathbb{L},\mu)$ are two valued fields with values in a 
common valued group $\Gamma$. They induce the norms $V_\nu$ and $V_\mu$ on $\K^n$ and
$\mathbb{L}$ respectively. Let $\phi:\K^n\to\mathbb{L}^m$ be a mapping (any kind),
such that $V_\nu\circ\phi\leqslant V_\mu$. Then there exists a map $\Gr(\phi):
\Gr_\nu(\K^n)\to\Gr_\mu(\mathbb{L}^m)$, sending $\Init_\nu(z)$ to
$z\mod\mathcal{P}_{\alpha}(\mathbb{L}^m)$ where $\alpha=V_\nu(z)$.
The association $\phi\mapsto\Gr(\phi)$ is functorial.
\end{lm}

The proof is identical to that of \cref{thm:funct}. One needs to replace the groups
$I_\alpha,I^+_\alpha$ with the corresponding $\mathcal{P}_\alpha,\mathcal{P}^+_\alpha$.
The following result mirrors the analog result for valued rings, \ie \cref{lm:sub}.
Its proof stems from the definitions.

\begin{lm}\label{lm:Gr_sub}
Consider a bijection $\phi:\K^n\to\K^n$, such that $V_\nu\circ\phi=V_\nu$. Then both
$\Gr(\phi)$ and $\Gr(\phi^{-1})$ are defined and are inverse to each other. Hence
\[\forall z\in\K^n\setminus\{0\},\:\Gr(\phi)(\Init_\nu(z))=\Init_\nu(\phi(z)).\]
\end{lm}

\begin{ex}\label{ex:triangular}
Consider the bijective linear mapping $\phi:\K^n\to\K^n$ defined by
$\phi(z)=(\phi_1(z),\ldots,\phi_n(z))$, where
\[\phi_i(z)=z_i+\sum_{j>i}a_{i,j}z_j,\quad a_{i,j}\in\Or_\nu.\]
We show that $\phi$ satisfies $V_\nu\circ\phi=V_\nu$. It is clearly invertible.
For any index $i$ and all $z\in\K^n\setminus\{0\}$
\[\nu(\phi_i(z))\leqslant\max\Big(\{\nu(z_i)\}\cup\{\nu(a_{i,j}z_j:j>i\}\Big)
\leqslant\max\{\nu(z_j):j\geqslant i\}\leqslant V_\nu(z).\]
Thus $V_\nu\circ\phi\leqslant V_\nu$.
For the converse inequality, fix $z\in\K^n\setminus\{0\}$ and let $i$ be the highest index
for which $V_\nu(z)=\nu(z_i)$. Thus, for all $j>i,\,\nu(a_{i,j}z_j)\leqslant
\nu(z_j)<\nu(z_i)$. By \cref{rke:basics} (4) we have $\nu(\phi_i(z))=V_\nu(z)$.
\end{ex}

\begin{rke}
Observe how $\Gr_\nu(\K^n)$ is constructed as a disjoint union, whereas $\gr_\nu(\K)$
is constructed as a direct sum. In \cite{Duc}, the author defines a notion of
\textit{annélloïde} (which could be roughly translated to \textit{ringoid}), which is
a disjoint union $\bigsqcup_{m\in M}A_m$, where $M$ is an abelian monoid, the $A_m$ are
abelian groups and there is a well-defined notion of homogeneous multiplication, \ie
compatible morphisms $A_m\times A_{m'}\to A_{m+m'}$. In such structures, one can multiply
homogeneous elements and add homogeneous elements of same degree.

We choose to keep $\gr_\nu(\K)$ as a graded algebra for the results concerning
substitutions of initial forms (\cf \cref{lm:sub}).
We also wish to keep $\Gr_\nu(\K^n)$ as a disjoint union as to give a more geometric
nature to this object, this is illustrated in \cref{prop:split_pres}.
\end{rke}

We end this paragraph by relating the graded module over $\K^n$ to the graded algebra
$\gr_\nu(\K)$.

\begin{prop}\label{prop:phase_presentation}
For any $\alpha\in\Gamma$ there is a canonical isomorphism of abelian groups
\[\Theta_\alpha:G_\alpha\xrightarrow{\simeq} g_\alpha^n.\]
They combine in canonical isomorphisms of $\kappa_\nu$ vector spaces
\[\widetilde{\Gr}_\nu(\K^n)\simeq\bigoplus_{\alpha\in\Gamma}(g_\alpha)^n
\simeq(\gr_\nu(\K))^n.\]
We also have a homogeneous bijection between ringoids
$\Gr_\nu(\K^n)\simeq\bigsqcup_{\alpha\in\Gamma}g_\alpha^n$.
\end{prop}

\begin{proof}
Consider $Z=\Init_\nu(z)\in G_\alpha\setminus\{0\}$. Thus $z=(z_1,\ldots,z_n)\in\K^n$ and
$V_\nu(z)=\max_i\nu(z_i)=\alpha$. Define
\[\Theta_\alpha(Z)=\theta=(\theta_1,\ldots,\theta_n)\quad\text{where}\quad
\theta_i:=\left\{\begin{array}{ll}
    \init_\nu(z_i) & \text{if }\nu(z_i)=\alpha \\
    0 & \text{if }\nu(z_i)<\alpha.
\end{array}\right.\]
The map $\Theta_\alpha$ is a group morphism. From the above construction, it is clear that
for any $a\in\Or_\nu^\times$ we have $\Theta(\overline{a}\cdot Z)=
\overline{a}\cdot\Theta_\alpha(Z)$. Let $Z,Z'\in G_\alpha$. Suppose $Z=\Init_\nu(z)$ and
$Z'=\Init_\nu(z')$. Define $w\in\K^n$ as follows
\[w_i=\left\{\begin{array}{ll}
    z_i+z_i' & \text{if }\nu(z_i)=\nu(z_i')=\nu(z_i+z_i')=\alpha \\
    z_i & \text{if } \nu(z_i')<\nu(z_i)=\alpha \\
    z_i' & \text{if } \nu(z_i)<\nu(z_i')=\alpha \\
    0 & \text{if } \nu(z_i+z_i')<\alpha.
\end{array}\right.\]
Observe that $y+y'-w\in\mathcal{P}_\alpha^+(\K^n)$, so that $Z+Z'=\Init_\nu(w)$.
By the properties and case work in \cref{prop:in_compute}, we have that
$\Theta_\alpha(\Init_\nu(w))=\Theta_\alpha(Z)+\Theta_\alpha(Z')$.

The morphism $\Theta_\alpha$ is surjective. Consider $\theta\in g_\alpha^n$. Without loss
of generality, suppose $\theta_i=\init_\nu(z_i)$ for $i=1,\ldots,m$ and $\theta_i=0$ for
$i=m+1,\ldots,n$. Set the vector $z=(z_1,\ldots,z_m,0,\ldots,0)\in\K^n$. We have
$\theta=\Theta_\alpha(Z)$ for $Z=\Init_\nu(z)$.

The morphism $\Theta_\alpha$ is injective. Indeed, by construction, if $Z\neq0$, then
$\Theta_\alpha(Z)\neq0$.
\end{proof}

We henceforth identify the phase space with $\bigsqcup_{\alpha\in\Gamma}g_\alpha^n$. Thus
a phase will be an element of type $(\alpha,\theta)$ where $\alpha\in\Gamma$ and
$\theta\in g_\alpha^n$.

\subsection{Splittings}
Abstract graded algebras can have a very complicated structure. We present our objects
in the particular case of valuations admitting a splitting or cross-section.

\begin{defn}
Let $(\K,\nu)$ be a valued field with value group $\Gamma_\nu$. A
\textit{multiplicative splitting} or \textit{cross-section} is a family of elements
$(t^\gamma)_{\gamma\in\Gamma_\nu}$ such that the mapping $\Gamma_\nu\to\K^\times,\:\gamma
\mapsto t^\gamma$ is a group morphism and $\nu(t^\gamma)=\gamma$ for all $\gamma\in\Gamma_\nu$. 
\end{defn}

Cross-sections are hard to find. For fields $\K$ which are $\aleph_1$-saturated (in the
model-theoretic sense), one can construct a splitting (\cf \cite[3.3.39, 3.3.40]{ADH} and
\cite[Thm. A.10, Variant A.11]{PC} for more ample details).

\begin{prop}\label{prop:split_exist}
Suppose that for any $n\in\N$ and $x\in\Or_\nu^\times$, the equation $X^n=x$ has a solution
in $\K$ (\eg $\K$ is algebraically closed, or perfect and henselian).
Then $(\K,\nu)$ admits a section.
\end{prop}

\begin{proof}
The condition of the proposition translates to $\Or_\nu^\times$ being divisible. Thus
$\Or_\nu^\times$ is an injective $\Z$-module. Consequently the short exact sequence
\[0\to\Or_\nu^\times\to\K^\times\to\Gamma_\nu\to0\]
is split and it yields a section.
\end{proof}

\begin{prop}\label{prop:split_pres}
Let $(\K,\nu)$ be a valued field with splitting $(t^\gamma)_{\gamma\in\Gamma_\nu}$.
\begin{enumerate}
\item The graded algebra $\gr_\nu(\K)$ has a presentation as a group algebra.
More precisely, the mapping
\[\varphi\,:\:\kappa_\nu\left[t^{\Gamma_\nu}\right]\to\gr_\nu(\K),\quad
\overline{a}\cdot[\gamma]\mapsto\init_\nu(at^\gamma)\]
is an isomorphism.
\item One has a natural bijection
\[\vartheta\,:\:\Gamma_\nu\times\kappa_\nu^n\to\Gr_\nu(\K^n),\quad
(\alpha,(\overline{z_1},\ldots,\overline{z_n}))\mapsto
\Init_\nu(z_1t^{\alpha},\ldots,z_nt^\alpha).\]
It is compatible with scalar multiplication by homogeneous elements of $\gr_\nu(\K)$,
where the multiplication on $\Gamma_\nu\times\kappa_\nu^n$ is defined as follows:
$\init_\nu(b)\cdot(\alpha,\overline{a})=(\alpha+\delta,\overline{abt^{-\delta}})$ where
$\delta=\nu(a)$.
\end{enumerate}
\end{prop}

\begin{proof}
Consider (1). The map $\varphi$ is injective: for a formal expression $a=\sum_{\gamma}
\overline{a_\gamma}[\gamma]$ is sent to $\sum_{\gamma}\init_\nu(at^\gamma)$,
but each term $\init_\nu(at^\nu)$ is of differing degree, thus
$\sum_{\gamma}\init_\nu(at^\gamma)=0$ if and only if
$\forall\gamma,\:\overline{a_\gamma}[\gamma]=0$, \ie $a=0$.

The map $\varphi$ is surjective. any element of $\gr_\nu(\K)$ is of the form
$\theta=\sum_i\init_\nu(r_i)$. Let $\gamma_i=\nu(r_i)$, so that
$\theta=\varphi\left(\sum_i\overline{r_it^{-\gamma_i}}[\gamma_i]\right)$.

Point (2) is a simple consequence of \cref{prop:phase_presentation} paired with
the fact that $\kappa_\nu$ is isomorphic to every $g_\gamma$ via
$\overline{a}\mapsto\init_\nu(at^\gamma)$.
\end{proof}

\begin{ex}
We continue \cref{ex:Hahn}. The above gives a formal proof of the fact that
$\gr_{\ord_t}(F[[t^\Gamma]])=F[t^\Gamma]$.
The isomorphism $\vartheta$ allows us to deconstruct any phase $\theta\in\Gr_\nu(\K^n)$.
Take any $z=(z_1,\ldots,z_n)\in\K^n$ and write $z_i=A_it^{\alpha_i}+o(t^{\alpha_i})$.
Thus $V_\nu(x)=\alpha=\max_i\alpha_i$. Construct the vector $B=(B_1,\ldots,B_n)\in F^n$
\[B_i=\left\{\begin{array}{ll}
    A_i & \text{if }\alpha_i=\alpha  \\
    0 & \text{if }\alpha_i<\alpha
\end{array}\right.\]
We write this as $z=Bt^\alpha+o(t^\alpha)$. Thus composing $\Init_\nu$ with $\vartheta^{-1}$
gives us a more explicit formula for the initial form of vectors: it that takes the vector
$Bt^\alpha+o(t^\alpha)$ and sends it to $(\alpha,B)$.
\end{ex}

\subsection{A lifting theorem}
For an ideal $I$ of $\K[x_1,\ldots,x_n]$ and $\alpha\in\Gamma$, we define
$\IN_{\underline{\alpha}}(I)$ as the ideal generated by
$\{\IN_{\underline{\alpha}}(f):\,f\in I\}$.

\begin{lm}\label{lm:INIT}
For any ideals $I,I_1,I_2$ of $\K[x_1,\ldots,x_n]$ we have
\begin{align*}
&\IN_{\underline{\alpha}}(I)\subseteq\IN_{\underline{\alpha}}(\sqrt{I})
\subseteq\sqrt{\IN_{\underline{\alpha}}(I)}
\quad\text{and} \\
&\IN_{\underline{\alpha}}(I_1)\cdot\IN_{\underline{\alpha}}(I_2)\subseteq
\IN_{\underline{\alpha}}(I_1\cap I_2)\subseteq
\IN_{\underline{\alpha}}(I_1)\cap\IN_{\underline{\alpha}}(I_2).
\end{align*}
\end{lm}

\begin{proof}
The first and last inclusions follow easily from the fact that if $I\subseteq J$, then
$\IN_{\underline{\alpha}}(I)\subseteq\IN_{\underline{\alpha}}(J)$.
Consider the second inclusion and take $F=\IN_{\underline{\alpha}}(f)
\in\IN_{\underline{\alpha}}(\sqrt{I})$, with $f\in\sqrt{I}$. Then $f^N\in I$ for some
$N\in\N,\,N\geqslant1$. Thus $F^N=\IN_{\underline{\alpha}}(f^N)
\in\IN_{\underline{\alpha}}(I)$, so that $F\in\sqrt{\IN_{\underline{\alpha}}(I)}$.
The third inclusion comes from the multiplicativity of $\IN_{\underline{\alpha}}$.
\end{proof}

\begin{rke}
It may very well be that $\IN_{\underline{\alpha}}(I)$ is not a radical ideal, even if
$I$ is a radical ideal. Consider for instance $f=x_1^N-1\in\K[x_1,\ldots,x_n]$, where
$\mathrm{char}(\K)\nmid N$ and $I=\langle f\rangle$. Set $\alpha>0$ and
$\underline{\alpha}=(\alpha,\ldots,\alpha)\in\Gamma_\nu^n$. Then clearly $I$ is a radical
ideal, yet $\IN_{\underline{\alpha}}(I)=\langle\IN_{\underline{\alpha}}(f)\rangle=
\langle\IN_{\underline{\alpha}}(x_1)^N\rangle$ is not.
\end{rke}

\begin{ex}\label{ex:principal}
For principal ideals $I=\langle f\rangle$, it is clear that for any $\alpha,\:
\IN_{\underline{\alpha}}(I)=\langle\IN_{\underline{\alpha}}(f)\rangle$.
\end{ex}

For an ideal $\mathcal{I}$ of $\gr_\nu(\K)[X_1,\ldots,X_n]$ homogeneous for its
$\underline{\alpha}=(\alpha,\ldots,\alpha)$ weighting, define
$\mathbb{V}_\alpha(\mathcal{I}):=\{Z\in g_\alpha^n:\:F(Z)=0,\text{ for all }F\in\mathcal{I}\}$.
We can formulate the tropical lifting theorem for phase tropicalizations.

\begin{thm}\label{thm:phase_trop}
For any affine algebraic set $X=\mathbb{V}(I)\subset\K^n$
\[\forall\alpha\in\Gamma_\nu,\quad \Init_\nu(X)_{\alpha}=
\{\alpha\}\times\mathbb{V}_\alpha(\IN_{\underline{\alpha}}(I))),
\quad\mathrm{where}\quad\underline{\alpha}=(\alpha,\ldots,\alpha).\]
\end{thm}

\begin{proof}
Consider $z=(z_1,\ldots,z_n)\in X$ with $V_\nu(z)=\alpha$ and $F=\IN_{\underline{\alpha}}(f)$
where $f\in I$. We are clearly in the case (1) of \cref{lm:sub} with $m=0$ and $g_i=z_i$,
thus
\[\IN_{\underline{\alpha}}(f)(\init_\nu(z_1),\ldots,\init_\nu(z_n))=0.\]
Having taken arbitrary $z\in X$ and $f\in I$, we have that $\Init_\nu(X)_\alpha\subset
\{\alpha\}\times\mathbb{V}_\alpha(\IN_{\underline{\alpha}}(I)))$.

The reverse inclusion is more tedious and we proceed in steps.
First we reduce to the case of $I$ being prime. By \cref{lm:INIT} we can replace the ideal
$I$ by its radical $\sqrt{I}$. Indeed $X=\mathbb{V}(\sqrt{I})$ and
\[\mathbb{V}_\alpha(\IN_{\underline{\alpha}}(I))=
\mathbb{V}_\alpha\left(\sqrt{\IN_{\underline{\alpha}}(I)}\right)=
\mathbb{V}_\alpha\left(\IN_{\underline{\alpha}}(\sqrt{I})\right).\]
We henceforth assume $I$ to be radical. Thus we can write $I=\bigcap_{k=1}^N\p_k$,
where the $\p_k$ are prime. Thus $X=\bigcup_{k=1}^N\mathbb{V}(\p_k)$, where the
$X_k=\mathbb{V}(\p_k)$ are the irreducible components of $X$ and so $\Init_\nu(X)=
\bigcup_{k=1}^N\Init_\nu(X_k)$. Furthermore, by \cref{lm:INIT} we also have
$\mathbb{V}_\alpha(\IN_{\underline{\alpha}}(I))=
\bigcup_{k=1}^N\mathbb{V}_\alpha(\IN_{\underline{\alpha}}(\p_k))$. Thus it suffices to show the
theorem for the $\p_k$. Without loss of generality we can thus assume $I$ to be prime,
or $X$ an irreducible affine scheme.

Consider any $\theta=(\theta_1,\ldots,\theta_n)\in
\mathbb{V}_\alpha(\IN_{\underline{\alpha}}(I))$.
Up to a permutation of variables we can assume that $\theta_{m+1}=\ldots=\theta_n=0$ and that
$\theta_1,\ldots,\theta_m$ are non-zero. We set the following change of variables
\[\begin{array}{rclrcl}
    y_1&=&x_1 & y_{m+1}&=&x_{m+1}+x_1 \\
    & \vdots & & & \vdots& \\
    y_m&=&x_m & y_n&=&x_n+x_1.
\end{array}\]
They define an isomorphism $\Phi:\K[y_1,\ldots,y_n]\to\K[x_1,\ldots,x_n]$.
It is clear that $\Phi$ satisfies the conditions of \cref{lm:sub}. Thus $\Phi$ yields an
isomorphism of graded algebras
\[\gr(\Phi):\gr_\nu(\K)[Y_1,\ldots,Y_n]\to\gr_\nu(\K)[X_1,\ldots,X_n]\]
with the grading that gives $X_i$ and $Y_i$ weight $\alpha$. Furthermore we have
\begin{equation}\label{eq:IN}
\Gr(\phi)(\mathbb{V}_\alpha(\IN_{\underline{\alpha}}(I))
=\mathbb{V}_\alpha(\gr_\nu(\Phi^{-1})(\IN_{\underline{\alpha}}(I)) \\
=\mathbb{V}_\alpha(\IN_{\underline{\alpha}}(\Phi^{-1}(I)).
\end{equation}
The ring morphism $\Phi$ induces a corresponding linear isomorphism $\phi:\K^n\to\K^n$.
Since $\phi$ fits within \cref{ex:triangular}, by \cref{lm:Gr_sub} $\phi$ induces a
bijection
\[\Gr(\phi):\Gr_\nu(\K^n)\xrightarrow{\sim}\Gr_\nu(\K^n)\]
that sends $\theta$ to $\vartheta$, where $\vartheta_i=\theta_i$ for $i=1,\ldots,m$ and
$\vartheta_i=\theta_1$ for $i=m+1,\ldots,n$. Furthermore we have
\begin{equation}\label{eq:GR}
\Gr(\phi)(\Init_\nu(X)) = \Init_\nu(\phi(X)).
\end{equation}
Furthermore, since $X=\mathbb{V}(I)$, we have $\phi(X)=\mathbb{V}(\Phi^{-1}(I))$.
Combining this last fact with \eqref{eq:IN} and \eqref{eq:GR}, we can assume that all the
entries of $\theta$ are non-zero. However this situation is already established in
\cite[Prop. 3.2.11, p. 108]{MS}. One thus obtains a point $z\in X\cap\mathbb{T}^n$
such that $\Init_\nu(z)=\theta\in\Init_\nu(X)$.
\end{proof}

\begin{rke}
One can transpose the above results to the set-up of \cite[Ch. 2]{MS} as follows.
For any $f=\sum_{u\in U}c_ux^u\in\K[x_1,\ldots,x_n]$ define
\[R_{\underline{\alpha}}(f):=\sum_{u\in U_0}\overline{c_ut^{-\nu(c_u)}}X^u
\in\kappa_\nu[X_1,\ldots,X_n],\quad
U_0:=\{u\in U:\,\nu(c_u)+u\cdot\underline{\alpha}=\nu_{\underline{\alpha}}(f)\}\]
and define $R_{\underline{\alpha}}(I)$ as the ideal
in $\kappa_\nu[X_1,\ldots,X_n]$ generated by $\{R_{\underline{\alpha}}(f):\,f\in I\}$.
At first glance the ideals structure $\kappa_\nu[X_1,\ldots,X_n]$ seems simpler than that of
$\gr_\nu(\K)[X_1,\ldots,X_n]$. However, $\gr_\nu(\K)$ is a saturated graded algebra, \ie
every homogeneous element is invertible. Furthermore we observe that we have an identity
\[t^{\nu_{\underline{\alpha}}(f)}R_{\underline{\alpha}}(f)(X_1,\ldots,X_n)=
\IN_{\underline{\alpha}}(f)(t^{\alpha_1}X_1,\ldots,t^{\alpha_n}X_n),\]
for any $f\in\K[x_1,\ldots,x_n]$. Thus the $\kappa_\nu$-algebra morphism
\[\kappa_\nu[X_1,\ldots,X_n]\to\gr_\nu(\K)[X_1,\ldots,X_n],\:
F(X_1,\ldots,X_n)\mapsto F(t^{\alpha_1}X_1,\ldots,t^{\alpha_n}X_n),\]
induces a bijection between their ideals of type $R_{\underline{\alpha}}(I)$ and
ideals of type $\IN_{\underline{\alpha}}(I)$, for the same ideal $I$.
\end{rke}

\begin{rke}\label{rke:dimension}
We continue the previous remark. We identify $g_\alpha$ with $\kappa_\nu$ via
$\init_\nu(a)\mapsto\overline{ay^{-1}}$ for any element $y\in\K$ such that $\nu(y)=\alpha$.
Using this identification and the $\kappa_\nu$-algebra morphism
\[\kappa_\nu[X_1,\ldots,X_n]\to\gr_\nu(\K)[X_1,\ldots,X_n],\:
F(X_1,\ldots,X_n)\mapsto F(yX_1,\ldots,yX_n),\]
we see that $\mathbb{V}_\alpha(\IN_{\underline{\alpha}}(I))$ can be identified with the
$\kappa_\nu$ algebraic set $\mathbb{V}(R_{\underline{\alpha}}(I))\subset\kappa_\nu^n$.
Thus $\Init_\nu(X)_\alpha$ can be seen as an algebraic set. We assume $I$ to be a radical
ideal. Thus its associated primes are all of equal dimension $d=\dim(X)$.
By \cref{lm:INIT} we can thus deduce that $\IN_{\underline{\alpha}}$ and
$R_{\underline{\alpha}}(I)$ are also radical and by \cite[Lemma 2.4.12, p. 70]{MS} we can
deduce that $R_{\underline{\alpha}}$ is of Krull dimension $d$. Thus $\Init_\nu(X)_\alpha$ is
an algebraic set of dimension $d$.
\end{rke}

\section{Level structure of phase tropicalizations}\label{sec:layered}

In this section we restrict to rank one valuations. We can thus assume that $\Gamma=\R$.
We do this in order to use universal Gröbner bases of an ideal. It is a theory
established only for rank one valuations, although a similar theory for valuations
of any finite rank should not be difficult to conceive. For instance one would need to
replace polyhedra with lex-polyhedra.

Our goal in the following is to study what $\Init_\nu(X)$ looks like when fibered over
$\Gamma$. Our main result could be likened to a Morse-Bott theorem for tropicalizations.

\begin{prop}\label{prop:crit_values}
Fix an ideal $I$ of $\K[x_1,\ldots,x_n]$. There is a finite set of values
$\beta_0<\ldots<\beta_r$ in $\Gamma$ such that $\forall\alpha\in(\beta_{i},\beta_{i+1})$
the initial ideal $\IN_{\underline{\alpha}}(I)$ is constant and homogeneous.
\end{prop}

\begin{defn}
We call the values $\beta_i$, the \textit{critical levels} of $I$.
\end{defn}

\begin{proof}
We first consider the homogenization $I^{\mathrm{hom}}$ of $I$. It is an ideal of
$\K[x_0,x_1,\ldots,x_n]$. By \cite[Cor. 2.5.11, p. 80]{MS} it has a universal Gröbner
basis $\{G_1,\ldots,G_s\}$.

For any vector $u=(u_1,\ldots,u_n)\in\R^n$ write $|u|=\sum_{i=1}^nu_i$.
The following argument is already present in the proof of \cite[Prop. 2.6.1]{MS}.
Consider $f=\sum_{u\in U}c_ux^u\in I$ and $\tilde{f}=\sum_{u\in U}c_ux^ux_0^{j_u}\in
I^{\mathrm{hom}}$ its homogenization, so that $\deg(f)=\max_{v\in U,\,c_v\neq0}|v|$
and $j_u=\deg(f)-|u|$. Thus
\[\IN_{(0,\underline{\alpha})}(\tilde{f})=\sum_{u\in U_0}\init_\nu(c_u)X^uX_0^{j_u}\]
where
\[U_0=\{u\in U:\:\nu(c_u)+\alpha|u|=\nu_{\underline{\alpha}}(f)\}.\]
Finally
\[\IN_{(0,\underline{\alpha})}(\tilde{f})|_{X_0=1}=\IN_{\underline{\alpha}}(f)\]
and so $\IN_{\underline{\alpha}}(I)$ is the image of $\IN_{(0,\underline{\alpha})}
(I^{\mathrm{hom}})$ in $\K[x_1,\ldots,x_n]$ obtained by setting $X_0=1$. If we set
$g_j=G_j(1,x_1,\ldots,x_n)$ we thus have $\IN_{\underline{\alpha}}(I)=
\big\langle\IN_{\underline{\alpha}}(g_1),\ldots,\IN_{\underline{\alpha}}(g_s)\big\rangle$
for all $\alpha$.

For all $j=1,\ldots,s$ set $g_j=\sum_{u\in U_j}c_u^jx^u$. For
$\IN_{\underline{\alpha}}(g_j)$ to be homogeneous, a sufficient condition
\footnote{This condition might be far from necessary.} is that all terms
$c_u^jx^u,\:u\in U_j$ of distinct weight $|u|$ be of distinct value under
$\nu_{\underline{\alpha}}$. For any $u,v\in U_j,\:|u|\neq|v|$ we have
\begin{align*}
\nu_{\underline{\alpha}}(c_u^jx^u)=\nu_{\underline{\alpha}}(c_v^jx^v) 
&\iff\nu(c_u^j)+\underline{\alpha}\cdot u=\nu(c_v^j)+\underline{\alpha}\cdot v \\
&\iff\nu(c_u^j/c_v^j)=\underline{\alpha}\cdot(v-u)=\alpha|v-u|\\
&\iff\alpha=\frac{\nu(c_u^j/c_v^j)}{|v|-|u|}.
\end{align*}
We set $V_j:=\left\{\frac{\nu(c_u^j/c_v^j)}{|v|-|u|}:\:u,v\in U_j,|u|\neq|v|\right\}$
and $V=\cup_{j=1}^sV_j$. Denote $V=\{\beta_0,\ldots,\beta_r\}$. For any
$\alpha\in\Gamma\setminus V$ we have $\IN_{\underline{\alpha}}(I)$ generated by
homogeneous elements. Between two consecutive $\beta_i$, the terms $c_u^jx^u$ for
which the $\nu_{\underline{\alpha}}$ value is minimal does not change,
thus the ideal $\IN_{\underline{\alpha}}(I)$ does not vary. Observe that the set
$V$ might contain a value $\beta_i$ for which $\IN_{\underline{\alpha}}(I)=
\IN_{\underline{\beta_i}}(I)$ for all $\alpha\in(\beta_{i-1},\beta_{i+1})$.
We may simply eliminate those values and thus obtain our critical levels.
\end{proof}

As an application, we retrieve the general picture of double-hyperbolic tropicalizations of
surfaces \cref{thm:surfaces}. We henceforth work with the field $\K$ of Hahn series and thus
identify $\Gr_\nu(\K^n)$ with $\R\times\C^n$. We recall briefly that for a variety
$X\subset\mathrm{SL}_2(\K)$, its double-hyperbolic tropicalization is
$\widehat{\varkappa}(\widetilde{\mathrm{VAL}}(X))$, where
$\widehat{\varkappa}(C)=(CC^*,C^*C)$ for a matrix $C\in\mathrm{SL}_2(\C)$.
The layered structure that is stated in \cref{thm:surfaces} stems from the general picture
of the valuation tropicalization $\Init_\nu(X)$ since it dominates the
double-hyperbolic tropicalization. First of all, $\mathrm{SL}_2(\K)$ is cut out by a single
equation $\det-1$, thus by \cref{ex:principal}, $\Init_\nu(\mathrm{SL}_2(\K))=\bigsqcup_\alpha
\mathbb{V}(\IN_{\underline{\alpha}}(\det-1))$. We clearly have
\[\IN_{\underline{\alpha}}(\det-1)=\left\{\begin{array}{ll}
    1 & \text{if } \alpha<0\\
    \det-1 & \text{if } \alpha=0\\
    \det & \text{if } \alpha>0.
\end{array}\right.\]
By directly using the method in the proof of \cref{prop:crit_values}, we see that
$\mathrm{SL}_2(\K)$ has only one critical level, $\beta=0$ and
\[\Init_\nu(\mathrm{SL}_2(\K))=\Big(\{0\}\times\mathrm{SL}_2(\C)\Big)
\cup\Big((0,+\infty)\times\{\det=0\}\Big).\]
Since $X\subset\mathrm{SL}_2(\K)$, the critical levels $0=\beta_0<\ldots<\beta_r$ of $X$
are non-negative. Let $I$ be the defining ideal of $X$. We know that the ideal
$\IN_{\underline{\alpha}}(I)$ is constant and homogeneous for
$\alpha\in(\beta_{i},\beta_{i+1})$. By \cref{rke:dimension}, this amounts to saying that its
double-hyperbolic tropicalization will be the union of cylinders $(\beta_{i},\beta_{i+1})\times
C_i$, for some complex projective curve $C_i$ and the fibers over the critical levels. By the
same remark, $\IN_{\underline{\beta_i}}(I)$ is an inhomogeneous ideal of $\C[x_1,x_2,x_3,x_4]$
of height $2$. Thus when projectivizing $\{\det=0\}$ to $Q_2(\C)$, the image of
$\mathbb{V}(\IN_{\underline{\beta_i}}(I))$ dominates $Q_2(\C)$. This amounts to saying that
above the positive critical levels, the double-hyperbolic tropicalizations are of the form
$\{\beta_i\}\times Q_2(\C)$. Finally, the fiber over the vertex at the level $\beta=0$
collapses to a point in the double-hyperbolic tropicalization.

We end our investigation, with a theorem about valuative tropicalizations of general surfaces
inside $\mathrm{SL}_2(\K)$. For any polynomial $f=\sum_{u\in U}c_ux^u\in\K[x_1,\ldots,x_n]$, we
define $\mathrm{Trop}(f)\cdot\alpha:=\max\{\nu(c_u)+\alpha|u|:\:u\in U,\:c_u\neq0\}$, the
\textit{tropical polynomial} of $f$. The graph of the function $\alpha\in\R\mapsto
\mathrm{Trop}(f)\cdot\alpha$ is a continuous, piecewise linear set of $\R^2$. The first
coordinates of the bends of this graph are the \textit{tropical roots} of $\mathrm{Trop}(f)$.

Consider an irreducible surface $X\subset\mathrm{SL}_2(\K)$. By Klein's theorem it is a
complete intersections, \ie it is cut out by one extra equation $f\in\K[x_1,\ldots,x_4]$.
It may happen that for some $\alpha>0,\:\IN_{\underline{\alpha}}(f)=G\cdot\det$. In this
case, we consider the largest $\ell$ for which such an $\alpha$ sits inside
$[\beta_\ell,\beta_{\ell+1})$ and lift $G$ to $g\in\K[x_1,\ldots,x_4]$. We replace $f$
with $f-g(\det-1)$. It may still happen that $\IN_{\underline{\alpha}}(f)=H\cdot\det$ but
for $\alpha$ inside some $[\beta_{\ell'},\beta_{\ell'+1})$ with $\ell'<\ell$. After a finite
number of adjustments we thus obtain $f$ such that $\forall\alpha>0,\:
\IN_{\underline{\alpha}}(f)\notin\langle\det\rangle$. We call such $f$ $\det$ free.

Finally observe that for any $u\in U$ the functions $\alpha\mapsto\mathrm{Trop}(c_ux^u)
\cdot\alpha=\nu(c_u)+\alpha|u|$ are either constant, for $u=0$, or increasing for $|u|>0$.
Thus the graph of $\alpha\mapsto\mathrm{Trop}(f)\cdot\alpha=
\max_{u\in U}\mathrm{Trop}(c_ux^u)\cdot\alpha$ is a non-decreasing function.

One can further simplify the expression of $f$ to a reduced function $\widetilde{f}$ such
that $\mathrm{Trop}(f)=\mathrm{Trop}(\widetilde{f})$. Consider $f=\sum_{u\in U}c_ux^u$, with
$c_u\neq0\iff u\in U$ and set
\[\widetilde{U}:=\{u\in U:\:\exists\alpha\in\R,\:
\mathrm{Trop}(f)\cdot\alpha=\mathrm{Trop}(c_ux^u)=\nu(c_u)+\alpha|u|\}.\]
We define $\widetilde{f}=\sum_{u\in\widetilde{U}}c_ux^u$ and furthermore decompose it in
homogeneous components $\widetilde{f}=\widetilde{f}_0+\cdots+\widetilde{f}_d$.
Since $f$ and $\widetilde{f}$ have the same tropical polynomial, they have the same tropical
roots. Furthermore, it is easy to check that they also have the same initial forms:
$\forall\alpha,\:\IN_{\underline{\alpha}}(f)=\IN_{\underline{\alpha}}(\widetilde{f})$.
It may happen that $\widetilde{f}_i=0$, for which we set
$\IN_{\underline{\alpha}}(\widetilde{f}_i)=0$ for all $\alpha$.
It is clear to see how the different $\IN_{\underline{\alpha}}(\widetilde{f})$ arise.
Call $\beta_1<\ldots<\beta_r$ the tropical roots of $\mathrm{Trop}(\widetilde{f})$ and set
$\beta_0=0$. There are degrees $0\leqslant d_0\leqslant\ldots\leqslant d_r$ such that
\begin{align*}
\IN_{\underline{\alpha}}(\widetilde{f}) &=
    \IN_{\underline{\alpha}}(\widetilde{f}_{d_i}),\: \forall\alpha\in(\beta_i,\beta_{i+1})
    \text{ and} \\
\IN_{\underline{\beta_i}}(\widetilde{f}) &=
    \IN_{\underline{\alpha}}(\widetilde{f}_{d_{i-1}})+\cdots+
    \IN_{\underline{\beta_i}}(\widetilde{f}_{d_i}),\:i=1,\ldots,r.
\end{align*}
One can simplify the expression of $\widetilde{f}$ even further. If we are working over the
field of Hahn series $\K$ one can replace every $c_ux^u,\:u\in\widetilde{U}$ with $\widehat{c_u}t^{\nu(c_u)}$ with $\widehat{c_u}\in\C$ and $c_u=\widehat{c_u}t^{\nu(c_u)}+
o(t^{\nu(c_u)})$. We denote the ensuing polynomial $\widehat f$. Since we have reduced $f$ to
$\widetilde{f}$, it is clear that $\widehat{f}=t^{\gamma_0}\widehat{f}_0+\cdots+
t^{\gamma_d}\widehat{f}_d$ with $\widehat{f}_i\in\C[x_1,\ldots,x_n]$. The following result
shows that valuative tropicalizations of general irreducible surfaces inside
$\mathrm{SL}_2(\K)$ are the same as those given by such simplified expressions.

\begin{prop}\label{prop:principal}
For general surfaces $X$, the positive critical levels of $X$ are the tropical roots of
$\mathrm{Trop}(f)$ and for all $\alpha>0,\:\Init_\nu(X)=
\mathbb{V}(\det,\IN_{\underline{\alpha}}(f))$.
\end{prop}

\begin{proof}
Consider a decomposition into homogeneous components $f=f_0+\cdots+f_d$.
For general coefficients $c_u$, it is clear that $f$ is $\det$ free and all possible sets
\[\mathbb{V}(\det,\IN_{\underline{\alpha}}(f))\subset\C^4,\:\alpha\in\Gamma,\]
are irreducible of dimension $2$ in $\C[x_1,\ldots,x_4]$. The main argument here is the fact
that there are only a finite amount of possible initial forms $\IN_{\underline{\alpha}}(f)$.
Since $\Init_\nu(X)$ is clearly included in $\mathbb{V}(\det,\IN_{\underline{\alpha}}(f))$
and is of dimension $2$, thus the two sets are equal.
\end{proof}

\begin{rke}
The above results can be generalized to other valued fields $\K$ equipped with a section.
For instance one could consider the field of Hahn series with real coefficients or the
field of $p$-adic numbers $\Q_p$.
\end{rke}

\appendix

\section{Another approach to lifting}\label{app:alt_proof}

In this appendix, we briefly sketch out an alternative proof of \cref{thm:phase_trop}, which
is independent of \cite[Prop. 3.2.11, p. 108]{MS}. We recall the statement: for any affine
algebraic set $X=\mathbb{V}(I)\subset\K^n$ and any level $\alpha\in\Gamma_\nu$, the set
$\Init_\nu(X)_{\alpha}=\{\Init_\nu(x)\::\:x\in X,\:V(x)=\alpha\}$ coincides with
$\{\alpha\}\times\mathbb{V}_\alpha(\IN_{\underline{\alpha}}(I)))$

\begin{itemize}
\item \textit{Step 0:} We show the inclusion $\Init_\nu(X)_\alpha\subset\{\alpha\}\times
\mathbb{V}_\alpha(\IN_{\underline{\alpha}}(I)))$. This has been done already in the proof of
\cref{thm:phase_trop}. Thus we now focus on the reverse inclusion. Fix
$(\init_\nu(z_1),\ldots,\init_\nu(z_n))\in\mathbb{V}_\alpha(\IN_{\underline{\alpha}}(I)))$
\item \textit{Step 1:} We show \cref{thm:phase_trop} when $X$ is hypersurface case.
The basic strategy consists in looking for solutions of the form $x_i=z_i+y_i$ where
$\nu(y_i)<\alpha$. We leave out the details of this process.
\item \textit{Step 2:} We reduce to the irreducible or prime ideal case (just like in the
first part of the proof of \cref{thm:phase_trop}).
\item \textit{Step 3:} We prove a very subtle Noether normalization theorem. We assume our
base field $\K$ to be infinite. Set a variety $X_0\hookrightarrow\A^n_\K$ of dimension $d$,
codimension $c=n-d$ and fix a point $x\in X_0$. One can find a general linear projection
$\pi:\A^n_\K\to\A^d_\K$, such that the induced morphism $X_0\to\pi(X_0)=X_d\subset\A^d_\K$
is finite. We can refine this process and show that $\pi$ can be factored into two linear
projections $\pi=p\circ q$ with $p:\A^n_\K\to\A^{d+1}_\K$ and $q:\A^{d+1}_\K\to\A^d_\K$, such
that $q^{-1}(q(x))=\{x\}$. In this situation $p(X_0)=X_{d-1}\subset\A^{d+1}_\K$ is a
hypersurface. Both $X_{d-1}$ and $X_d$ are irreducible varieties. This is an affine version
of the refinement of the projective version of Noether's normalization that can be found in
\cite[Proposition (2.32), p. 38]{Mu}.
\item \textit{Step 4:} We apply the previous step to $X_0=X$. Set $I_0,I_{d-1},I_d$ the
respective ideals of $X_0,X_{d-1}$ and $X_d$. By functoriality we obtain the following
\[\begin{tikzcd}
\A^{n}_{\gr_\nu(\K)} \ar[r,"\Init_\nu(p)"] & \A^{d-1}_{\gr_\nu(\K)}
    \ar[r,"\Init_\nu(q)"] & \A^{d}_{\gr_\nu(\K)} \\
\mathbb{V}_\alpha(\IN_{\underline{\alpha}}(I_0))
    \ar[u,closed] \ar[r,"\widetilde{\Init_\nu(p)}"] &
        \mathbb{V}_\alpha(\IN_{\underline{\alpha}}(I_{d-1}))
        \ar[u,closed] \ar[r,"\widetilde{\Init_\nu(q)}"] &
            \mathbb{V}_\alpha(\IN_{\underline{\alpha}}(I_d)) \ar[u,equal] \\
\Init_\nu(X_0) \ar[u,closed] \ar[r,"\widehat{\Init_\nu(p)}"] &
    \Init_\nu(X_{d-1}) \ar[u,equal] \ar[r,"\widehat{\Init_\nu(q)}"] &
        \Init_\nu(X_d) \ar[u,equal]
\end{tikzcd}\]
It is important to maintain the genericity conditions on the projections $\Init_\nu(p)$ and
$\Init_\nu(q)$ by translating them into genericity conditions on $p$ and $q$.
\item \textit{Step 5:} From the hypersurface case, we can conclude that
$\mathbb{V}_\alpha(\IN_{\underline{\alpha}}(I_{d-1}))=\Init_\nu(X_{d-1})$. We conclude via a
simple diagram chase, considering that $\widehat{\Init_\nu(p)}$ is surjective and step 4.
\end{itemize}

\section{\texorpdfstring{$\mathrm{SL}_2$}{SL2} valuation and
the diffeomorphism}\label{app:diffeo}


Let $\tilde R_h\colon\mathrm{SL}_2(\C)\to\mathrm{SL}_2(\C)$ be a diffeomorphism defined in terms
of the polar decomposition by $PU\mapsto P^hU$. Consider an $\mathrm{SL}_2(\C)$-valued series $A_t$
converging for large real $t$. We want to compute the limit of $\tilde R_{\log(t)^{-1}}(A_t)$
as $t\to+\infty$.

First, we will find the unitary part of the limit. For this, we need to know the
expression of the spherical coamoeba map, \ie of $U$ for a matrix
$C=PU\in\mathrm{SL}_2(\C)$. Note that $P^2=CC^*$ and $\varkappa^\circ{C}=U=P^{-1}C$,
thus we are left with expressing the inverse of the square root of $P^2$.
This can be done by the Cayley-Hamilton formula:
\[P^2-\mathrm{tr}(P)P+1=0.\]
Thus, $P^{-1}=(P^{-2}+1)/\mathrm{tr}(P)=[\![(C^*)^{-1}C^{-1}+1]\!]$.
Therefore, $\varkappa^{\circ}(C)=[\![C+(C^*)^{-1}]\!]$.

To compute the limit of the unitary part of $A_t$, we need to rethink the operation
of taking the inverse of a unimodular matrix -- namely, $C^{-1}=\mathrm{tr}(C)-C$
for $C$ with $\det(C)=1$, again by Cayley-Hamilton formula. The expression
$C^{adj}=\mathrm{tr}(C)-C$ extends to all two-by-two matrices as a linear map known as
\textit{adjugate}, satisfying the core property $CC^{adj}=C^{adj}C=\det(C)C$.

Thus, the limit of the unitary part of $A_t$ expanded asymptotically as
$t^{\alpha}B+o(t^{\alpha})$, with $B$ being a complex matrix, is easy to compute:
\[\lim_{t\to\infty}\varkappa^\circ(A_t)=
\lim_{t\to\infty}[\![t^{\alpha}(B+(B^*)^{adj})+o(t^{\alpha})]\!]=
[\![B+(B^*)^{adj}]\!]\in\mathrm{SU}(2).\]

Note that in the asymptotic expansion $A_t=t^{\alpha}B+o(t^{\alpha})$ the exponent
$\alpha$ cannot be negative since $1=det(A_t)=t^{2\alpha}det(B)+o(t^{2\alpha})$.
Moreover, if $\alpha=0$ then $\det(B)=1$ and $\det(B)=0$ otherwise.
In the case $\alpha=0$, we see that $A_t$ itself converges to $B$,
and thus the limit of $R_{\log(t)^{-1}}(A_t)$ is the same as of
$R_{\log(t)^{-1}}(B)$, which is again $[\![B+(B^*)^{adj}]\!]$. 

Therefore, we are left with the case $\alpha>0$, when $A_t$ itself does not converge
in the vector space of two-by-two matrices. To compute the limit of the
$\log(t)^{-1}$ power of the Hermitian part $P(t)\in\mathbb{H}^3$ of $A_t$,
what we do is polar decomposing $\mathbb{H}^3$ itself with respect to the matrix
$1$ being the center of the decomposition. Namely, first we compute the limit of the
distance from $(P(t))^{(log(t))^{-1}}$ to the identity matrix. Note that such a distance
is computed as the logarithm of the top eigenvalue of $A_tA^*(t)$ divided by $2\log(t)$.
Since $A_tA^*(t)$ is unimodular and diverges, this eigenvalue is asymptotically equal
to $\mathrm{tr}(A_tA^*(t))=t^{2\alpha}\mathrm{tr}(BB^*)+o(t^{\alpha})$. Thus,
the distance from the Hermitian part of $R_{\log(t)^{-1}}(A_t)$ to $I_2$ converges
to $\alpha>0$.

This was the radial component of the limit in $\mathbb{H}^3$. To find the angular component
in $S^2=\partial\mathbb{H}^{3}$ consisting of rank one hermitian matrices considered
up to a scalar multiple, we simply compute the limit of $A_tA^*(t)$ normalized
by its trace (which is necessarily non-vanishing). This gives the point
$BB^*/\mathrm{tr}(BB^*)\in S^2$ as the limit of the angular part.

Now we have to combine this angular part $BB^*/\mathrm{tr}(BB^*)$ with the radial part
$\alpha>0$ to get a point in $\mathbb{H}^3$. What we need to find is such a Hermitian
unimodular matrix which has the same eigenvectors as $BB^*$, and the eigenvalue of
the vector spanning the image of $BB^*$ should be $e^\alpha$ (this is due to the distance
$\alpha$ to the identity matrix). The desired expression is
$P=e^{-\alpha}I_2+(e^\alpha-e^{-\alpha})BB^*/\mathrm{tr}(BB^*)$.

Combining it with the unitary part, we get
\begin{align*}
\lim_{t\to\infty}(R_{\log(t)^{-1}})(A_t)
&=(e^{-\alpha}I_2+(e^\alpha-e^{-\alpha})BB^*/\mathrm{tr}(BB^*))[\![B+(B^*)^{adj}]\!] \\
&=[\![e^{-\alpha}B+(e^\alpha-e^{-\alpha})BB^*B/\mathrm{tr}(BB^*)\\
&+e^{-\alpha}(B^*)^{adj}+(e^\alpha-e^{-\alpha})BB^*(B^*)^{adj}/\mathrm{tr}(BB^*)]\!].
\end{align*}
This expression can be simplified since $B$ is a rank-one matrix, \ie the last term is
vanishing due to $B^*(B^*)^{adj}=0$, and $BB^*B/\mathrm{tr}(BB^*)=B$ because if we write
$B=vw$ where $v$ is a column vector and $w$ is a row vector, $\mathrm{tr}(BB^*)=tr(vww^*v^*)
=\mathrm{tr}(v^*vww^*)=\mathrm{tr}(|v|^2|w|^2)=|v|^2|w|^2$ and
$BB^*B=vww^*v^*vw=|v|^2|w|^2vw$. Therefore, the final formula for
$A_t=t^\alpha B+o(t^\alpha)$ is
\[\lim_{t\rightarrow\infty}(R_{\log(t)^{-1}})(A_t)=
[\![e^{\alpha}B+e^{-\alpha}(B^*)^{adj}]\!]\in\mathrm{SL}_2(\C),\]
which is also compatible with the case $\alpha=0$ that we considered separately. 

The final remark here is that the expression $[\![e^{\alpha}B+e^{-\alpha}(B^*)^{adj}]\!]$
for $\alpha>0$ (and so $\det(B)=0$) contains exactly the same information as the pair
$(\alpha,[B]_{\mathbb{R}_{>0}})$, where $[B]_{\mathbb{R}_{>0}}$ denotes the class of matrices
equal to $B$ up to a positive real multiple. More formally, there is a diffeomorphism
$\Psi:(0,+\infty)\times\mathcal{S}\to\mathrm{SL}_2(\C)/\mathrm{SU}(2)$ given by
\[(\alpha,[B]_{\mathbb{R}_{>0}})\mapsto [\![e^{\alpha}B+e^{-\alpha}(B^*)^{adj}]\!].\]

It is possible to write an explicit (and quite cumbersome) formula for its inverse.
However, from the geometric perspective, the invertibility of the above map is not hard
to see. Namely, a matrix $C\in\mathrm{SL}_2(\C)\backslash\mathrm{SU}(2)$, is distinct from its
inverse-conjugate $(C^*)^{-1}$. Take a 2-plane $H$ in the space of matrices passing through
$0,\,C$ and $(C^*)^{-1}$. The projectivisation of this plane intersects the quadric surface
$\det=0$ in two conjugate points $[B]_{\mathbb{C}^*}$ and $[(B^*)^{adj}]_{\mathbb{C}^*}$.
Thus, we have an alternative basis in the plane $H$, \ie the two matrices $B_0$ and
$(B_0^*)^{adj}$, and our initial point $C\in H$ can be written uniquely as
$C=aB_0+b(B_0^*)^{adj}$ for a pair of non-zero complex numbers $a$ and $b$.
Note that since $\det(C)=1$ the arguments of $a$ and $b$ are opposite.
Now we simply solve for $\alpha\in\mathbb{R}^*,\,c\in\mathbb{C}^*$, where $B=cB_0$
and $r>0$ the following equation
\[aB_0+b(B_0^*)^{adj}=r(e^{\alpha}cB+e^{-\alpha}\overline{c} (B^*)^{adj}),\]
which gives the inverse of $C$ under $\Psi$, where we may switch the roles of $B$
and $(B^*)^{adj}$ if $\alpha$ has turned out to be negative (in that case we also replace
$\alpha$ with $-\alpha$).

\end{document}